\documentclass[10pt]{amsart}
\usepackage{amssymb,amstext,amsmath,amscd,amsthm,amsfonts,enumerate,latexsym,stmaryrd,multicol,geometry,graphicx,mathrsfs}
\usepackage[dvipdfmx]{hyperref}
\usepackage[usenames]{color}
\usepackage[all]{xy}
\newtheorem{thm}{Theorem}[section]
\newtheorem{lem}[thm]{Lemma}
\newtheorem{prop}[thm]{Proposition}
\newtheorem{cor}[thm]{Corollary}
\theoremstyle{definition}
\newtheorem{dfn}[thm]{Definition}
\newtheorem{ques}[thm]{Question}

\newtheorem{rem}[thm]{Remark}

\newtheorem{ex}[thm]{Example}

\theoremstyle{remark}

\newtheorem*{claim*}{Claim}
\newtheorem*{ac}{Acknowlegments}
\newtheorem*{conv}{Convention}

\numberwithin{equation}{thm}
\def\add{\operatorname{add}}
\def\ann{\operatorname{ann}}
\def\bA{\mathbf{A}}
\def\bB{\mathbf{B}}
\def\bC{\mathbf{C}}
\def\bD{\mathbf{D}}
\def\bE{\mathbf{E}}
\def\bF{\mathbf{F}}
\def\bG{\mathbf{G}}
\def\bH{\mathbf{H}}
\def\bI{\mathbf{I}}
\def\bJ{\mathbf{J}}
\def\bK{\mathbf{K}}
\def\bL{\mathbf{L}}
\def\bM{\mathbf{M}}
\def\C{\mathcal{C}}

\def\cok{\operatorname{Cok}}
\def\db{\operatorname{D^b}}
\def\depth{\operatorname{depth}}
\def\dx{\operatorname{dx}}
\def\dm{\operatorname{D^{--}}}
\def\ds{\operatorname{D^{sg}}}
\def\e{\operatorname{e}}
\def\edim{\operatorname{edim}}

\def\Ext{\mathrm{Ext}}
\def\G{\mathcal{G}}

\def\ge{\geqslant}
\def\grade{\operatorname{grade}}
\def\gt{\operatorname{gt}}
\def\h{\mathrm{H}}
\def\Hom{\mathrm{Hom}}
\def\height{\operatorname{ht}}
\def\I{\operatorname{I}}
\def\im{\operatorname{Im}}
\def\jac{\operatorname{jac}}
\def\k{\mathrm{K}}

\def\le{\leqslant}
\def\level{\operatorname{level}}
\def\lten{\otimes^\mathbf{L}}
\def\m{\mathfrak{m}}
\def\mod{\operatorname{mod}}
\def\N{\mathbb{N}}
\def\n{\mathfrak{n}}
\def\nf{\operatorname{NF}}
\def\ospec{\operatorname{Ospec}}
\def\p{\mathfrak{p}}
\def\pd{\operatorname{pd}}
\def\r{\operatorname{r}}
\def\rank{\operatorname{rank}}
\def\soc{\operatorname{Soc}}
\def\syz{\mathrm{\Omega}}
\def\T{\mathcal{T}}
\def\thick{\operatorname{thick}}
\def\tr{\operatorname{Tr}}
\def\U{\mathcal{U}}
\def\udim{\operatorname{udim}}
\def\V{\mathrm{V}}
\def\X{\mathcal{X}}
\def\XX{\boldsymbol{X}}
\def\xx{\boldsymbol{x}}
\def\Y{\mathcal{Y}}
\def\YY{\boldsymbol{Y}}
\def\yy{\boldsymbol{y}}
\def\ZZ{\mathbb{Z}}
\begin{document}
\title[Uniformly dominant local rings and Orlov spectra of singularity categories]{Uniformly dominant local rings and\\
Orlov spectra of singularity categories}
\author{Ryo Takahashi}
\address{Graduate School of Mathematics, Nagoya University, Furocho, Chikusaku, Nagoya 464-8602, Japan}
\email{takahashi@math.nagoya-u.ac.jp}
\urladdr{https://www.math.nagoya-u.ac.jp/~takahashi/}
\thanks{The author was partly supported by JSPS Grant-in-Aid for Scientific Research 23K03070}
\dedicatory{Dedicated to Professor Kei-ichi Watanabe on the occasion of his eightieth birthday}
\begin{abstract}
We define a uniformly dominant local ring as a commutative noetherian local ring with an integer $r$ such that the residue field is built from any nonzero object in the singularity category by direct summands, shifts and at most $r$ mapping cones.
We find sufficient conditions for uniform dominance, by which we show Burch rings and local rings with quasi-decomposable maximal ideal are uniformly dominant.
For a uniformly dominant excellent equicharacteristic isolated singularity, we get an upper bound of the Orlov spectrum of the singularity category.
We prove uniform dominance is preserved under basic operations, and give techniques to construct uniformly dominant local rings.
An application of our methods to local rings with decomposable maximal ideal is provided as well.
\end{abstract}
\maketitle
\section{Introduction}

As one of the main results of their celebrated paper \cite{BFK}, Ballard, Favero and Katzarkov prove the following theorem.
Denote by $\jac(-)$ and $\ell\ell(-)$ the Jacobian ideal and the Loewy length, respectively.

\begin{thm}[Ballard--Favero--Katzarkov \cite{BFK}]\label{43}
Let $(R,\m,k)$ be a complete equicharacteristic local hypersurface of dimension $d$ such that $k$ is algebraically closed and of characteristic $0$.
Suppose that $R$ has an isolated singularity.
Put $J=\jac R$ and $l=\ell\ell(R/J)$.
Then, all the nonzero objects of the singularity category $\ds(R)$ have generation time at most $2(d+2)l-1$, i.e., the equality $\ds(R)=\langle X\rangle_{2(d+2)l}$ holds for all $0\ne X\in\ds(R)$.
\end{thm}

\noindent
To explain the terminology introduced by Ballard, Favero and Katzarkov \cite{BFK}, let $\T$ be a triangulated category.
The {\em generation time} of an object $X$ of $\T$ is by definition the least integer $n$ such that every object of $\T$ can be built out of $X$ by taking finite direct sums, direct summands, shifts and at most $n$ extensions.
The {\em Orlov spectrum} of $\T$ is defined as the set of finite generation times of objects of $\T$, and the {\em ultimate dimension} of $\T$ is defined to be the supremum of this set.
The conclusion of Theorem \ref{43} especially says that the singularity category $\ds(R)$ has finite Orlov spectrum and ultimate dimension at most $2(d+2)l-1$.

In this paper, we define a {\em uniformly dominant} local ring as a local ring $R$ such that there exists an integer $r$ such that in $\ds(R)$ the residue field of $R$ can be built out of each nonzero object by taking direct summands, shifts and at most $r$ mapping cones.
We call the infimum of such integers $r$ the {\em dominant index} of $R$.
When $R$ is a uniformly dominant local ring, $R$ is a dominant local ring in the sense of \cite{dlr}, whence $R$ is Tor-friendly and Ext-friendly in the sense of \cite{AINS}, in particular, the Auslander--Reiten conjecture holds for $R$, and moreover, the thick subcategories of $\ds(R)$ are completely classified under some assumptions; see \cite{dlr} for the details.

For a local ring $R$, we denote by $\ann\ds(R)$ the {\em annihilator} of $\ds(R)$, that is to say, the set of elements $a$ of $R$ such that for each object $X$ of $\ds(R)$ the multiplication morphism $X\xrightarrow{a}X$ in $\ds(R)$ is zero; note that this is an ideal of $R$.
For a finitely generated $R$-module $M$, we denote by $\nu(M)$ the minimal number of generators of $M$, and by $\syz^nM$, where $n\in\N$, the $n$th syzygy of $M$ in a minimal free resolution of $M$.
Motivated by Theorem \ref{43}, we shall prove the following theorem.

\begin{thm}[Corollaries \ref{40} and \ref{28}]\label{33}
Let $(R,\m,k)$ be a local ring of depth $t$.
Put $s=1$ when $t=0$, and put $s=2^{\edim R}$ when $t>0$.
Then the following statements hold true.
\begin{enumerate}[\rm(1)]
\item
If the syzygy $\syz^{t+1}k$ (resp. $\syz^tk$) is a direct summand of a finite direct sum of copies of $\syz^{t+2}k$, then $R$ is a uniformly dominant local ring with dominant index at most $s(2t+3)-1$ (resp. $s(2t+4)-1$).
\item
Assume that $R$ is excellent, equicharacteristic and has an isolated singularity.
Let $J$ be an $\m$-primary ideal contained in $\ann\ds(R)$.
Put $m=\nu(J)$ and set $l=\ell\ell(R/J)$.
Suppose that $R$ is uniformly dominant with dominant index $n$.
Then every nonzero object of $\ds(R)$ has generation time at most $(n+1)(m-t+1)l-1$.
In particular, $\ds(R)$ has finite Orlov spectrum and has ultimate dimension at most $(n+1)(m-t+1)l-1$.
\end{enumerate}
\end{thm}

\noindent
It is a basic fact that if a local ring $(R,\m,k)$ of depth $t$ is a singular hypersurface, then $\syz^tk$ is isomorphic to $\syz^{t+2}k$.
Relating to this, Dao, Kobayashi and Takahashi \cite{burch} show that if $R$ is a hypersurface, then it is Burch and that if $R$ is a singular Burch ring, then $\syz^tk$ is a direct summand of $\syz^{t+2}k$.
In this paper, we prove that if $\m$ is quasi-decomposable in the sense of Nasseh and Takahashi \cite{fiber}, then $\syz^{t+1}k$ is a direct summand of $\syz^{t+2}k$ (Proposition \ref{10}).
Thus, the assumption of Theorem \ref{33}(1) is satisfied if either $\m$ is quasi-decomposable or $R$ is a singular Burch ring, paticularly if $R$ is a singular hypersurface.
As such an ideal $J$ as in Theorem \ref{33}(2), one can always take $\ann\ds(R)$, and can even take $\jac R$ if $R$ is a complete Cohen--Macaulay local ring and $k$ is perfect.
Therefore, Theorem \ref{33} considerably extends Theorem \ref{43} in terms of providing a finite uniform bound of the generation times of nonzero objects of the singularity category (the bound is itself looser).

Combining Theorems \ref{43}, \ref{33}, what is stated around them, and the examples of Burch rings and local rings with quasi-decomposable maximal ideal given in Example \ref{37}, and summarizing them, we get the following diagram of implications, where $(R,\m,k)$ is a local ring of depth $t$ and $k$ is assumed to be infinite for simplicity.
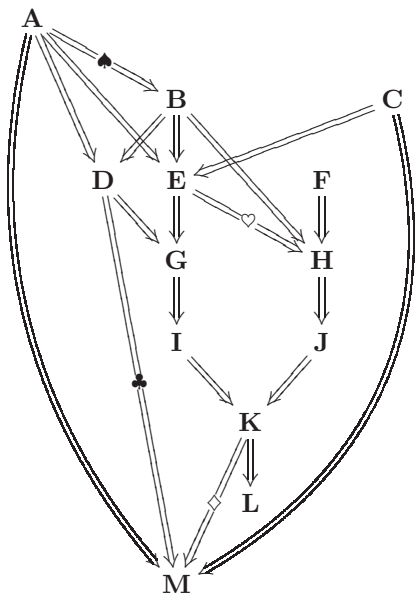
\begin{multicols}{3}
$\xymatrix@R1.48pc@C1pc{
\bA\ar@{=>}@/_30pt/[rrddddddd]\ar@{=>}[rdd]\ar@{=>}[rrdd]\ar@{=>}[rrd]|-\spadesuit\\
&&\bB\ar@{=>}[ld]\ar@{=>}[d]\ar@{=>}[rrdd]&&&\bC\ar@{=>}[llld]\ar@{=>}@/^40pt/[llldddddd]\\
&\bD\ar@{=>}[rddddd]|-\clubsuit\ar@{=>}[rd]&\bE\ar@{=>}[d]\ar@{=>}[rrd]|-\heartsuit&&\bF\ar@{=>}[d]\\
&&\bG\ar@{=>}[d]&&\bH\ar@{=>}[d]\\
&&\bI\ar@{=>}[rd]&&\bJ\ar@{=>}[ld]\\
&&&\bK\ar@{=>}[d]\ar@{=>}[ldd]|-\diamondsuit\\
&&&\bL\\
&&\bM
}$
\columnbreak

\ 
$\begin{array}{rl}
\bA=&\text{$R$ is regular}\\
\bB=&\text{$R$ deforms to an $(\mathrm{A}_1)$-singularity of dimension $1$}\\
\bC=&\text{$\widehat R$ is a rational surface singularity}\\
\bD=&\text{$R$ is a hypersurface}\\
\bE=&\text{$R$ is Cohen--Macaulay with minimal multiplicity}\\ \bF=&\text{$R$ is a fiber product of local rings over $k$}\\
\bG=&\text{$R$ is Burch}\\
\bH=&\text{$\m$ is quasi-decomposable}\\
\bI=&\text{$\syz^tk$ is a summand of a direct sum of copies of $\syz^{t+2}k$}\\
\bJ=&\text{$\syz^{t+1}k$ is a summand of a direct sum of copies of $\syz^{t+2}k$}\\
\bK=&\text{$R$ is uniformly dominant}\\
\bL=&\text{$R$ is dominant}\\
\bM=&\text{the generation times of nonzero objects of $\ds(R)$ form a}\\
&\text{finite subset of $\N$ (so $\ds(R)$ has finite Orlov spectrum)}\smallskip
\\
\spadesuit=&\text{if $R$ has dimension at least $2$}\\
\heartsuit=&\text{if $R$ is not a hypersurface}\\
\clubsuit=&\text{if $R$ is a complete equicharacteristic isolated singularity}\\
&\text{and $k$ is algebraically closed of characteristic $0$}\\
\diamondsuit=&\text{if $R$ is an excellent equicharacteristic isolated singularity}\\
\end{array}$
\end{multicols}

\noindent
The implication $(\bC\Rightarrow\bE)$ is shown to hold true by Artin \cite{Ar}, Sato and Takagi \cite{ST}, $(\bD,\bE\Rightarrow\bG\Rightarrow\bI)$ by Dao, Kobayashi and Takahashi \cite{burch}, $(\bE,\bF\Rightarrow\bH)$ by Nasseh and Takahashi \cite{fiber}, and $(\bD\Rightarrow\bM)$ by Ballard, Favero and Katzarkov \cite{BFK}.
In this paper we do prove $(\bH\Rightarrow\bJ)$ and $(\bI,\bJ\Rightarrow\bK\Rightarrow\bM)$.
The implication $(\bC\Rightarrow\bM)$ is a consequence of $(\bC\Rightarrow\bK\Rightarrow\bM)$ as $\bC$ implies $R$ is an isolated singularity.
All the other implications are clear.
Also, in this paper we present examples showing that $(\bH\nLeftarrow\bJ)$ and $(\bI,\bJ\nLeftarrow\bK)$; see Examples \ref{44} and \ref{45}.

We explore how uniform dominance is preserved under basic operations as well, and get the theorem below.
Thanks to this, various local rings turn out to be uniformly dominant; see Corollaries \ref{31}, \ref{13} and \ref{32}.

\begin{thm}[Theorem \ref{5} and Corollary \ref{11}]\label{39}
Let $(R,\m,k)$ be a local ring.
Let $x$ be an $R$-regular element in $\m$.
If $R/(x)$ is uniformly dominant with dominant index $n$, then $R$ is uniformly dominant with dominant index at most $2n+1$.
Conversely, if $R$ is uniformly dominant with dominant index $n$ and if $x$ is not in $\m^2$, then $R/(x)$ is uniformly dominant with dominant index at most $2n+1$.
In particular, uniform dominance ascends and descends along the completion map and a formal power series extension map of local rings.
\end{thm}

We mention an application of our methods.
The following theorem is a main result of \cite{fiber} concerning the structure of syzygies over a local ring with decomposable maximal ideal.

\begin{thm}[Nasseh--Takahashi \cite{fiber}]\label{42}
Let $(R,\m,k)$ be a local ring with $\m$ decomposable.
Then $\m$ is a direct summand of $\syz^3M\oplus\syz^4M\oplus\syz^5M$ for every finitely generated $R$-module $M$ with infinite projective dimension.
\end{thm}

The methods which we develop to explore uniformly dominant local rings give rise to the following theorem as a byproduct, which improves the above Theorem \ref{42}.

\begin{thm}[Theorem \ref{36}]\label{41}
Let $(R,\m,k)$ be a local ring with $\m$ decomposable.
Let $M$ be a finitely generated $R$-module of infinite projective dimension.
Then $\m$ is a direct summand of $\syz^3M\oplus\syz^4M$.
If $\m$ is not a direct summand of $\syz^5M$ or $\syz^6M$, then the $\m$-adic completion $\widehat R$ of $R$ is an $(\mathrm{A}_1)$-singularity of dimension one, i.e., $\widehat R\cong S/(xy)$ for some complete regular local ring $S$ of dimension two with a regular system of parameters $x,y$.
\end{thm}

\noindent
This theorem is well illustrated by taking into account the fact that if $R=k[\![x,y]\!]/(xy)$ and $M=R/(x)$, then $\m$ is not a direct summand of $\syz^iM$ for all nonnegative integers $i$.

This paper is organized as follows.
In Section 2, we investigate the structure of syzygies over a local ring with decomposable maximal ideal, and show Theorem \ref{41}.
In Section 3, we prove a technical but essential theorem on generation times of syzygies (Theorem \ref{19}).
In Section 4, as another application we find out sufficient conditions for a syzygy of the residue field to be finitely built out of a module of infinite projective dimension, and count the number of necessary extensions (Theorem \ref{9}).
In Section 5, we state the precise definition of a uniformly dominant local ring, and apply the theorem obtained in the previous section to the singularity category to prove Theorem \ref{33}.
In Section 6, we give a proof of Theorem \ref{39} and by using this theorem, we provide various methods to produce uniformly dominant local rings.

We close the section by stating our convention which is adopted throughout the rest of this paper.

\begin{conv}
By $\N$ we denote the set of nonnegative integers $\{0,1,2,\dots\}$.
We assume that all subcategories are nonempty and strictly full, that all rings are commutative and noetherian, and that all modules are finitely generated.
Each object $X$ of a category $\C$ may be identified with the (strictly full) subcategory of $\C$ given by $X$, which consists of those objects of $\C$ which are isomorphic to $X$.
Unless otherwise specified, we let $R$ be a local ring of depth $t$ with maximal ideal $\m$ and residue field $k$.
We denote by $(-)^\ast$ the $R$-dual functor $\Hom_R(-,R)$.
We let $\widehat R$ stand for the ($\m$-adic) completion of $R$.
The projective dimension of an $R$-module $M$ is denoted by $\pd_RM$.
A chain complex $X=(\cdots\to X_i\to X_{i-1}\to\cdots)$ is regarded as the cochain complex whose $(-i)$th component is $X_i$ for each $i\in\ZZ$.
When we say $(X,\partial)$ is a complex, $X$ is a complex with $i$th differential map $\partial^i$ for each $i\in\ZZ$.
Subscripts/superscripts may be omitted if they are clear from the context.
\end{conv}

\section{Syzygies over a local ring with decomposable maximal ideal}

In this section, we study syzygies of $R$-modules in the case where the maximal ideal $\m$ is decomposable as an $R$-module.
First of all, let us recall the precise definitions of syzygies, transposes and Betti numbers.

\begin{dfn}
Let $M$ be an $R$-module.
We denote by $\nu_R(M)$ the minimal number of generators of $M$, i.e., $\nu_R(M)=\dim_k(M\otimes_Rk)$.
Let $(F,\partial)$ be a minimal free resolution of $M$.
For each $n>0$ the {\em $n$th syzygy} $\syz_R^nM$ of $M$ is defined as the image of $\partial_n$, and we put $\syz_R^0M=M$.
Set $\beta_n^R(M)=\nu_R(\syz_R^nM)$ and call it the {\em $n$th Betti number} of $M$.
The {\em transpose} $\tr_RM$ of $M$ is defined as the cokernel of the $R$-dual map $\partial_1^\ast$.  
The $n$th syzygy and transpose of $M$ are uniquely determined up to isomorphism, since so is a minimal free resolution of $M$.
\end{dfn}

We investigate the structure of syzygies of direct summands of the maximal ideal $\m$ of the local ring $R$.

\begin{prop}\label{3}
Let $I$ and $J$ be nonzero ideals of $R$ such that $\m=I\oplus J$.
The following statements hold.
\begin{enumerate}[\rm(1)]
\item
The ideal $I$ is a direct summand of $\syz_RJ$.
The ideal $J$ is a direct summand of $\syz_RI$.
\item
If $R/J$ is not a discrete valuation ring, then $\m$ is a direct summand of $\syz_R^2I$ and of $\syz_R^3J$.
\item
If $R/I$ is not a discrete valuation ring, then $\m$ is a direct summand of $\syz_R^2J$ and of $\syz_R^3I$.
\end{enumerate}
\end{prop}

\begin{proof}
First of all, note that $I$ and $J$ have the structures of modules over $R/J$ and $R/I$, respectively, since $IJ=0$.
Applying \cite[Lemma 3.2]{fiber}, we observe that there are isomorphisms of $R$-modules:
$$
\begin{array}{l}
\syz_RI\cong J^{\oplus\beta_0^{R/J}(I)}\oplus\syz_{R/J}I,\qquad
\syz_R\syz_{R/J}I\cong J^{\oplus\beta_1^{R/J}(I)}\oplus\syz_{R/J}^2I,\\
\syz_RJ\cong I^{\oplus\beta_0^{R/I}(J)}\oplus\syz_{R/I}J,\qquad
\syz_R\syz_{R/I}J\cong I^{\oplus\beta_1^{R/I}(J)}\oplus\syz_{R/I}^2J.
\end{array}
$$
As $I$ and $J$ are nonzero ideals of $R$, the integers $\beta_0^{R/J}(I)$ and $\beta_0^{R/I}(J)$ are positive.
Hence the ideals $I$ and $J$ are direct summands of the syzygies $\syz_RJ$ and $\syz_RI$, respectively.
Thus assertion (1) follows.
We have
$$
\syz_R^2I
\cong(\syz_RJ)^{\oplus\beta_0^{R/J}(I)}\oplus\syz_R\syz_{R/J}I
\cong I^{\oplus\beta_0^{R/I}(J)\beta_0^{R/J}(I)}\oplus J^{\oplus\beta_1^{R/J}(I)}\oplus (\syz_{R/I}J)^{\oplus\beta_0^{R/J}(I)}\oplus\syz_{R/J}^2I.
$$
Since the ideal $I$ is isomorphic to $\m/J$, we have equalities $\beta_1^{R/J}(I)=\beta_1^{R/J}(\m/J)=\beta_2^{R/J}(k)$.
If $R/J$ is not a discrete valuation ring, then $\beta_2^{R/J}(k)$ is positive (note here that $R/J$ is not a field as $J\ne\m$), and $\m=I\oplus J$ is a direct summand of $\syz_R^2I$.
The module $\syz_R^3J=\syz_R^2(\syz_RJ)$ is isomorphic to $(\syz_R^2I)^{\oplus\beta_0^{R/I}(J)}\oplus\syz_R^2\syz_{R/I}J$, which contains $\syz_R^2I$ as a direct summand.
Hence $\m$ is a direct summand of $\syz_R^3J$ if $R/J$ is not a discrete valuation ring.
Thus assertion (2) follows.
Swapping $I$ and $J$ in assertion (2), we have that assertion (3) holds.
\end{proof}

Here we review some basic properties of a local ring with decomposable maximal ideal.
For the details, we refer the reader to \cite[Lemma 3.1]{O}, \cite[Fact 3.1]{fiber} and \cite[Fact 2.1]{gorfib}.

\begin{lem}\label{6}
\begin{enumerate}[\rm(1)]
\item
The maximal ideal $\m$ of $R$ is decomposable if and only if there exist two local rings $A$ and $B$ with residue field $k$ such that $R$ is isomorphic to the fiber product $A\times_kB$ of $A$ and $B$ over $k$.
\item
Suppose that $\m$ is decomposable, that is to say, that there exist nonzero ideals $I,J$ of $R$ such that $\m=I\oplus J$.
Then, an isomorphism  $R\cong R/I\times_kR/J$ and an equality $\depth R=\min\{\depth R/I,\,\depth R/J,\,1\}$ hold.
An $R$-module has infinite projective dimension if and only if it has projective dimension at least two.
\end{enumerate}
\end{lem}

In the lemma below we provide several properties of transposes, which will often be used in this paper.

\begin{lem}\label{1}
The following statements hold true for every $R$-module $M$.
\begin{enumerate}[\rm(1)]
\item
Suppose that $M$ has no nonzero free summand.
Then one has an isomorphism $M^\ast\cong\Hom_R(M,\m)$.
If moreover $\Ext_R^1(M,R)=0$, then one also has an isomorphism $\Ext_R^1(M,\m)\cong\Hom_R(M,k)$.
\item
One has $\syz^2M\cong(\tr M)^\ast\cong\Hom(\tr M,\m)$.
If $\Ext^1(\tr M,R)=0$, then $\Ext^1(\tr M,\m)\cong\Hom(\tr M,k)$.
\item
Let $I$ be a proper ideal of $R$.
Then $\tr_RM/I\tr_RM\cong\tr_{R/I}(M/IM)\oplus(R/I)^{\oplus\beta_1^R(M)-\beta_1^{R/I}(M/IM)}$.
\item
One has that $M$ is a free $R$-module if and only if the equality $\tr M=0$ holds.
\end{enumerate}
\end{lem}

\begin{proof}
Let $(F,f)=(\cdots\xrightarrow{f_4}F_3\xrightarrow{f_3}F_2\xrightarrow{f_2}F_1\xrightarrow{f_1}F_0\to0)$ be a minimal free resolution of $M$.

(1) Applying the functor $\Hom(M,-)$ to the natural exact sequence $0\to\m\xrightarrow{a}R\to k\to0$, we get an exact sequence $0\to\Hom(M,\m)\xrightarrow{b}M^\ast\xrightarrow{c}\Hom(M,k)\xrightarrow{d}\Ext^1(M,\m)\to\Ext^1(M,R)$.
The assumption that $M$ has no nonzero free summand guarantees that every homomorphism $M\to R$ factors through $a$.
Hence $b$ is an isomorphism, and therefore $c$ is a zero map.
If $\Ext^1(M,R)=0$, then the map $d$ is an isomorphism.

(2) Taking $(-)^\ast$ of the exact sequence $F_0^\ast\xrightarrow{f_1^\ast}F_1^\ast\to\tr M\to0$, we get an exact sequence $0\to(\tr M)^\ast\to F_1\xrightarrow{f_1}F_0\to M\to0$, so $(\tr M)^\ast\cong\syz^2M$.
By \cite[Lemma 4.2]{arg}, $\tr M$ has no nonzero free summand.
Assertion (1) implies that $(\tr M)^\ast\cong\Hom(\tr M,\m)$, and that $\Ext^1(\tr M,\m)\cong\Hom(\tr M,k)$ if $\Ext^1(\tr M,R)=0$.

(3) There is a free resolution $(G,g)$ of $M/IM$ over $R/I$ such that $g_1=f_1\otimes_RR/I$.
Let $(H,h)$ be a minimal free resolution of $M/IM$ over $R/I$.
There exist a split-exact $R/I$-complex $(E,e)$ and an isomorphism $G\cong H\oplus E$ of $R/I$-complexes; see \cite[Proposition 1.1.2]{A}.
An isomorphism $\Hom_{R/I}(G,R/I)\cong\Hom_{R/I}(H,R/I)\oplus\Hom_{R/I}(E,R/I)$ of $R/I$-modules is induced.
We obtain isomorphisms of $R/I$-modules:
$$
\begin{array}{l}
\tr_RM/I\tr_RM
\cong\cok(f_1^\ast\otimes_RR/I)
\cong\cok\Hom_{R/I}(g_1,R/I)\\
\phantom{\tr_RM/I\tr_RM\cong\cok(f_1^\ast\otimes_RR/I)}
\cong\cok\Hom_{R/I}(h_1,R/I)\oplus\cok\Hom_{R/I}(e_1,R/I).
\end{array}
$$
The module $\cok\Hom_{R/I}(h_1,R/I)$ is isomorphic to $\tr_{R/I}(M/IM)$.
Since $G_0,H_0$ are free $R/I$-modules of the same rank, we have $E_0=0$.
Hence $e_1$ is a map from $E_1$ to the zero module.
We see that $\cok\Hom_{R/I}(e_1,R/I)$ is isomorphic to $\Hom_{R/I}(E_1,R/I)$.
There are isomorphisms $\Hom_{R/I}(G_1,R/I)\cong\Hom_{R/I}(F_1\otimes_RR/I,R/I)\cong(R/I)^{\oplus\beta_1^R(M)}$ and $\Hom_{R/I}(H_1,R/I)\cong(R/I)^{\oplus\beta_1^{R/I}(M/IM)}$.
Therefore, the $R/I$-module $\Hom_{R/I}(E_1,R/I)$ is isomorphic to the free $R/I$-module $(R/I)^{\oplus\beta_1^R(M)-\beta_1^{R/I}(M/IM)}$.
We get an isomorphism as in the assertion.

(4) The exact sequence $F_0^\ast\xrightarrow{f_1^\ast}F_1^\ast\to\tr M\to0$ implies that if $M$ is free, then $F_1=0$, and $\tr M=0$.
Conversely, assume $\tr M$ is zero.
Then the map $f_1^\ast$ is surjective.
Since $\im f_1$ is contained in $\m F_0$, the module $F_1^\ast=\im(f_1^\ast)$ is contained in $\m F_1^\ast$.
By Nakayama's lemma, $F_1^\ast$ is zero, and so is $F_1$.
Therefore $M$ is free.
\end{proof}

Now we study the structure of syzygies of modules over a local ring with decomposable maximal ideal.

\begin{prop}\label{2}
Let $I$ and $J$ be nonzero ideals of $R$ such that $\m=I\oplus J$.
Let $M$ be an $R$-module.
\begin{enumerate}[\rm(1)]
\item
There are isomorphisms $\syz_R^2M\cong\Hom_R(\tr M,I)\oplus\Hom_R(\tr M,J)$, and
$$
\begin{array}{l}
\Hom_R(\tr M,I)\cong\syz_{R/J}^2(M/JM)\oplus I^{\oplus\beta_1^R(M)-\beta_1^{R/J}(M/JM)},\\
\Hom_R(\tr M,J)\cong\syz_{R/I}^2(M/IM)\oplus J^{\oplus\beta_1^R(M)-\beta_1^{R/I}(M/IM)}.
\end{array}
$$
In particular, there is an isomorphism
$$
\syz_R^2M\cong\syz_{R/I}^2(M/IM)\oplus\syz_{R/J}^2(M/JM)\oplus I^{\oplus\beta_1^R(M)-\beta_1^{R/J}(M/JM)}\oplus J^{\oplus\beta_1^R(M)-\beta_1^{R/I}(M/IM)}.
$$
\item
If $M$ has projective dimension at least two over $R$, then either $I$ or $J$ is a direct summand of $\syz_R^3M$.
\end{enumerate}
\end{prop}

\begin{proof}
(1) Lemma \ref{1}(2) yields $\syz^2M\cong\Hom_R(\tr M,\m)\cong\Hom_R(\tr M,I)\oplus\Hom_R(\tr M,J)$.
As $I$ is an $R/J$-module, $\Hom_R(\tr M,I)$ is isomorphic to $\Hom_{R/J}(\tr M/J\tr M,I)$.
By Lemma \ref{1}(3), we get isomorphisms
$$
\begin{array}{l}
\Hom_{R/J}(\tr M/J\tr M,I)
\cong\Hom_{R/J}(\tr_{R/J}(M/JM)\oplus(R/J)^{\oplus\beta_1^R(M)-\beta_1^{R/J}(M/JM)},I)\\
\phantom{\Hom_{R/J}(\tr M/J\tr M,I)}
\cong\Hom_{R/J}(\tr_{R/J}(M/JM),I)\oplus I^{\oplus\beta_1^R(M)-\beta_1^{R/J}(M/JM)}.
\end{array}
$$
As $I\cong\m/J$, the module $\Hom_{R/J}(\tr_{R/J}(M/JM),I)$ is isomorphic to $\Hom_{R/J}(\tr_{R/J}(M/JM),\m/J)$, which is isomorphic to $\syz_{R/J}^2(M/JM)$ by Lemma \ref{1}(2).
We get an isomorphism $\Hom_R(\tr M,I)\cong\syz_{R/J}^2(M/JM)\oplus I^{\oplus\beta_1^R(M)-\beta_1^{R/J}(M/JM)}$.
Swapping $I$ and $J$, we have $\Hom_R(\tr M,J)\cong\syz_{R/I}^2(M/IM)\oplus J^{\oplus\beta_1^R(M)-\beta_1^{R/I}(M/IM)}$.

(2) It is seen from assertion (1) that there exist isomorphisms of $R$-modules:
$$
\begin{array}{l}
\syz_R^3M\cong\syz_R(\syz_R^2M)\cong\syz_R\syz_{R/I}^2(M/IM)\oplus\syz_R\syz_{R/J}^2(M/JM)\\
\phantom{\syz_R^3M\cong\syz_R(\syz_R^2M)\cong\syz_R\syz_{R/I}^2(M/IM)}
\oplus (\syz_RI)^{\oplus\beta_1^R(M)-\beta_1^{R/J}(M/JM)}\oplus (\syz_RJ)^{\oplus\beta_1^R(M)-\beta_1^{R/I}(M/IM)}.
\end{array}
$$
Making use of \cite[Lemma 3.2]{fiber}, we observe that there are isomorphisms of $R$-modules:
$$
\begin{array}{l}
\syz_R\syz_{R/I}^2(M/IM)\cong I^{\oplus\beta_2^{R/I}(M/IM)}\oplus\syz_{R/I}^3(M/IM),\qquad
\syz_RI\cong J^{\oplus\beta_0^R(I)}\oplus\syz_{R/J}I,\\
\syz_R\syz_{R/J}^2(M/JM)\cong J^{\oplus\beta_2^{R/J}(M/JM)}\oplus\syz_{R/J}^3(M/JM),\qquad
\syz_RJ\cong I^{\oplus\beta_0^R(J)}\oplus\syz_{R/I}J.
\end{array}
$$
Hence $I^{\oplus a}\oplus J^{\oplus b}$ is a direct summand of $\syz^3M$, where $a=\beta_2^{R/I}(M/IM)+\beta_0^R(J)(\beta_1^R(M)-\beta_1^{R/I}(M/IM))$ and $b=\beta_2^{R/J}(M/JM)+\beta_0^R(I)(\beta_1^R(M)-\beta_1^{R/J}(M/JM))$.
Suppose that $a=b=0$.
Then $\beta_2^{R/I}(M/IM)=\beta_2^{R/J}(M/JM)=0$ and $\beta_1^{R/I}(M/IM)=\beta_1^{R/J}(M/JM)=\beta_1^R(M)$, since $\beta_0^R(I),\beta_0^R(J)$ are positive.
It follows from (1) that $\syz_R^2M=0$, which means $\pd_RM<2$.
The assertion follows by taking the contraposition.
\end{proof}

Applying the above proposition to the residue field, we get a remarkable property of its second syzygy.

\begin{cor}\label{7}
If the maximal ideal $\m$ of $R$ is decomposable, then $\m$ is a direct summand of $\syz_R^2k$.
\end{cor}

\begin{proof}
Take ideals $I,J\ne0$ with $\m=I\oplus J$.
Proposition \ref{2}(1) implies that $I^{\oplus a}\oplus J^{\oplus b}$ is a direct summand of $\syz^2k$, where $a=\beta_1^R(k)-\beta_1^{R/J}(k)$ and $b=\beta_1^R(k)-\beta_1^{R/I}(k)$.
Since $\m=I\oplus J$ and $I\cong\m/J$, we have that $a=\nu(\m)-\nu(\m/J)=\nu(\m)-\nu(I)=\nu(J)>0$, and similarly $b>0$.
Hence $\m$ is a direct summand of $\syz^2k$.
\end{proof}

It is natural to ask whether the converse of Corollary \ref{7} holds.
The example below says that it is negative.

\begin{ex}\label{44}
Let $k$ be a field, and let $R=k[x,y]/(x^3,x^2y,y^2)$.
Then the maximal ideal $\m=(x,y)$ of the artinian local ring $R$ is not decomposable, but $\syz^1k=\m$ is a direct summand of $\syz^2k$.
In fact, define a map $F:\m\to R^2$ by $F(ax+by)=\binom{-ay}{ax+by}$ for $a,b\in R$.
We have $(x,y)\binom{-ay}{ax+by}=-axy+axy+by^2=0$, so that $F(ax+by)$ is in the kernel of the map $(x,y):R^{\oplus2}\to R$, which is equal to $\syz^2k$.
Hence $F$ factors through the inclusion map $p:\syz^2k\hookrightarrow R^{\oplus2}$ and induces a map $f:\m\to\syz^2k$.
Let $G$ be the composite map of $p$ with the projection $(0,1):R^{\oplus2}\twoheadrightarrow R$.
As $p$ factors via the inclusion map $q:\m^{\oplus2}\hookrightarrow R^{\oplus2}$, the map $G$ factors via the inclusion map $r:\m\hookrightarrow R$ and induces a map $g:\syz^2k\to\m$.
We have that $gf(ax+by)=g(\binom{-ay}{ax+by})=ax+by$, which means that $gf$ is the identity map.
Consequently, $f$ is a split monomorphism, and the claim follows.
\end{ex}

In the corollary below, we obtain further information about modules over a local ring with decomposable maximal ideal.
The second assertion of the corollary gives rise to a sufficient condition for the maximal ideal to be a direct summand of the third syzygy of each nonfree module.

\begin{cor}
Let $I$ and $J$ be nonzero ideals of $R$ such that $\m=I\oplus J$.
Let $M$ be an $R$-module.
\begin{enumerate}[\rm(1)]
\item
If $\pd_RM$ is at most one, then so do $\pd_{R/I}(M/IM)$ and $\pd_{R/J}(M/JM)$.
\item
Assume $M$ is not free.
If $\depth R/I=0$, then $I$ is a direct summand of $\syz_R^3M$.
If $\depth R/J=0$, then $J$ is a direct summand of $\syz_R^3M$.
If $\depth R/I=\depth R/J=0$, then $\m$ is a direct summand of $\syz_R^3M$.
\end{enumerate}
\end{cor}

\begin{proof}
(1) Since $\pd_RM$ is at most one, we have $\syz_R^2M=0$ and hence $\syz_{R/I}^2(M/IM)=\syz_{R/J}^2(M/JM)=0$ by Proposition \ref{2}(1).
This means that both $\pd_{R/I}(M/IM)$ and $\pd_{R/J}(M/JM)$ are at most one.

(2) In view of Propopsition \ref{2}(1), the syzygy $\syz_R^2M$ is isomorphic to the direct sum of $\Hom_R(\tr M,I)$ and $\Hom_R(\tr M,J)$, while these are isomorphic to $\syz_{R/J}^2(M/JM)\oplus I^{\oplus\beta_1^R(M)-\beta_1^{R/J}(M/JM)}$ and $\syz_{R/I}^2(M/IM)\oplus J^{\oplus\beta_1^R(M)-\beta_1^{R/I}(M/IM)}$, respectively.
Lemma \ref{1}(4) implies that $\tr M$ is nonzero.
We make three steps.

(i) We consider the case where $R/I$ has depth zero.
Then so does $\m/I$, which is isomorphic to $J$.
The module $\Hom_R(\tr M,J)$ is nonzero; see \cite[Proposition 1.2.3(b)]{BH}.
If $\syz_{R/I}^2(M/IM)$ is zero, then $\beta_1^R(M)-\beta_1^{R/I}(M/IM)=\beta_1^R(M)>0$, and $J$ is a direct summand of $\Hom_R(\tr M,J)$, whence $I$ is a direct summand of $\syz_R\Hom_R(\tr M,J)$ by Proposition \ref{3}(1).
If $\syz_{R/I}^2(M/IM)$ is nonzero, then $\syz_R\syz_{R/I}^2(M/IM)$ is isomorphic to $I^{\oplus\beta_2^{R/I}(M/IM)}\oplus\syz_{R/I}^3(M/IM)$ by \cite[Lemma 3.2]{fiber} and $\beta_2^{R/I}(M/IM)>0$, whence $I$ is a direct summand of $\syz_R\Hom_R(\tr M,J)$.
Since the syzygy $\syz_R\Hom_R(\tr M,J)$ is a direct summand of $\syz_R^3M$, so is the ideal $I$.

(ii) Swapping $I$ and $J$ in the argument (i), we observe that in the case where the local ring $R/J$ has depth zero, the ideal $J$ of $R$ is a direct summand of both of the syzygies $\syz_R\Hom_R(\tr M,I)$ and $\syz_R^3M$.

(iii) By (i) and (ii), we conclude that if both $R/I$ and $R/J$ have depth zero, then $\m=I\oplus J$ is a direct summand of $\syz_R\Hom_R(\tr M,J)\oplus\syz_R\Hom_R(\tr M,I)$, which is isomorphic to the third syzygy $\syz_R^3M$.
\end{proof}

Let us state the main result of this section.
The theorem below improves \cite[Theorem A]{fiber}, which is one of the two main results of \cite{fiber}.
The second assertion of the theorem says if the maximal ideal is decomposable, except a certain particular case, it is a direct summand of either the fifth or sixth syzygy of a given module.

\begin{thm}\label{36}
Suppose that the maximal ideal $\m$ of the local ring $R$ is decomposable as an $R$-module.
Let $M$ be an $R$-module such that $\pd_RM\ge2$.
Then one has that $\pd_RM=\infty$, and the following statements hold.
\begin{enumerate}[\rm(1)]
\item
The maximal ideal $\m$ is a direct summand of $\syz_R^3M\oplus\syz_R^4M$.
\item
One of the following three statements holds.
\begin{enumerate}[\rm(a)]
\item
The maximal ideal $\m$ is a direct summand of $\syz_R^5M$.
\item
The maximal ideal $\m$ is a direct summand of $\syz_R^6M$.
\item
The completion of $R$ is an $(\mathrm{A}_1)$-singularity of dimension one, that is, there exist a complete regular local ring $S$ of dimension two and a regular system of parameters $x,y$ of $S$ such that $\widehat R\cong S/(xy)$.
\end{enumerate}
\end{enumerate}
\end{thm}

\begin{proof}
By assumption, there exist nonzero ideals $I$ and $J$ of $R$ such that the equality $\m=I\oplus J$ holds.

(1) In view of Proposition \ref{2}(2), we may assume that $I$ is a direct summand of $\syz^3M$.
In this case, $\syz I$ is a direct summand of $\syz^4M$.
By Proposition \ref{3}(1), the ideal $J$ is a direct summand of $\syz I$.
It follows that $J$ is a direct summand of $\syz^4M$.
We now conclude that $\m=I\oplus J$ is a direct summand of $\syz^3M\oplus\syz^4M$.

(2) Let us consider the case where $R/J$ is not a discrete valuation ring.
Then Proposition \ref{3}(2) implies that $\m$ is a direct summand of $\syz^2I$ and $\syz^3J$.
Also, either $I$ or $J$ is a direct summand of $\syz^3M$ by Proposition \ref{2}(2).
We thus have either that $\syz^2I$ is a direct summand of $\syz^5M$ or that $\syz^3J$ is a direct summand of $\syz^6M$.
Therefore, $\m$ is a direct summand of either $\syz^5M$ or $\syz^6M$.
By symmetry, this statement holds true as well in the case where $R/I$ is not a discrete valuation ring.
Now, suppose that both $R/I$ and $R/J$ are discrete valuation rings.
By Lemma \ref{6}(1), the ring $R$ is isomorphic to the fiber product $R/I\times_kR/J$.
It is observed from \cite[Corollary 2.7]{gorfib} and \cite[Theorem A]{syz} that there exist a complete regular local ring $S$ of dimension two and a regular system of parameters $x,y$ of $S$ such that the completion $\widehat R$ of $R$ is isomorphic to $S/(xy)$.
\end{proof}

\begin{rem}
Theorem \ref{36} is well illustrated by taking into account the fact that if $R=k[\![x,y]\!]/(xy)$ and $M=R/(x)$, then $\m$ is not a direct summand of $\syz^iM$ for all nonnegative integers $i$.
Indeed, the $R$-module $\syz^iM$ is isomorphic to either $R/(x)$ or $R/(y)$.
The $2$-generated $R$-module $\m$ cannot be a direct summand of a cyclic $R$-module.
\end{rem}

\section{A fundamental theorem}

In this section, we state and prove the most general fundamental theorem in the paper.
We begin with showing a lemma on syzygies and transposes.
The proof of the second assertion of the lemma contains analogous arguments as in the proof of Lemma \ref{1}(3).

\begin{lem}\label{15}
\begin{enumerate}[\rm(1)]
\item
Let $0\to L\to M\to N\to0$ be an exact sequence of $R$-modules.
Then for each nonnegative integer $n$ there exist exact sequences 
$0\to\syz^nL\to\syz^nM\oplus R^{\oplus a}\to\syz^nN\to0$, $0\to\syz^{n+1}N\to\syz^nL\oplus R^{\oplus b}\to\syz^nM\to0$ and $0\to\syz^{n+1}M\to\syz^{n+1}N\oplus R^{\oplus c}\to\syz^nL\to0$, where $a,b,c$ are nonnegative integers.
\item
Let $M$ be an $R$-module.
Then $n:=\beta_0^R(M)-\beta_1^R(\tr M)\ge0$ and $M\cong\tr(\tr M)\oplus R^{\oplus n}$.
\end{enumerate}
\end{lem}

\begin{proof}
(1) Let $0\to L\to M\xrightarrow{\pi}N\to0$ be an exact sequence.
We prove the lemma step by step.

(i) The horseshoe lemma provides an exact sequence $0\to\syz^nL\to\syz^nM\oplus R^{\oplus a}\to\syz^nN\to0$ with $a\in\N$.

(ii) Take an exact sequence $0\to\syz N\to R^{\oplus p}\xrightarrow{\varepsilon}N\to0$ with $p=\nu(N)$.
The pullback diagram of $\pi$ and $\varepsilon$ produces an exact sequence $0\to\syz N\to L\oplus R^{\oplus p}\to M\to0$.
In a similar way we get an exact sequence $0\to\syz M\to\syz N\oplus R^{\oplus q}\xrightarrow{\omega}L\oplus R^{\oplus p}\to0$ with $q=\nu(M)$.
There is an exact sequence $0\to L\xrightarrow{\eta}L\oplus R^{\oplus p}\to R^{\oplus p}\to0$.
The pullback diagram of $\omega,\eta$ gives an exact sequence $0\to\syz M\to K\to L\to0$ and an isomorphism $K\oplus R^{\oplus p}\cong\syz N\oplus R^{\oplus q}$.
We have $q=\nu(M)\ge\nu(N)=p$ as $\pi$ is surjective.
It follows from \cite[Corollary 1.16]{LW} that $K\cong\syz N\oplus R^{\oplus q-p}$.
We obtain an exact sequence $0\to\syz M\to\syz N\oplus R^{\oplus r}\to L\to0$ with $r=q-p\ge0$.

(iii) Combining (i) and (ii) gives rise to two exact sequences $0\to\syz^{n+1}N\to\syz^nL\oplus R^{\oplus b}\to\syz^nM\to0$ and $0\to\syz^{n+1}M\to\syz^{n+1}N\oplus R^{\oplus c}\to\syz^nL\to0$ of $R$-modules, where $b,c\in\N$.

(2) Let $(F,f)$ be a minimal free resolution of $M$.
Then there is an exact sequence $F_0^\ast\xrightarrow{f_1^\ast}F_1^\ast\to\tr M\to0$.
Extend this to a free resolution $(G,g)$ of $\tr M$, so that $(G_1\xrightarrow{g_1}G_0)$ coincides with $(F_0^\ast\xrightarrow{f_1^\ast}F_1^\ast)$.
Let $(H,h)$ be a minimal free resolution of $\tr M$.
It follows from \cite[Proposition 1.1.2]{A} that there are a split-exact complex $(E,e)$ and a complex isomorphism $G\cong H\oplus E$.
This induces a complex isomorphism $G^\ast\cong H^\ast\oplus E^\ast$, and we obtain module isomorphisms $M\cong\cok f_1\cong\cok(g_1^\ast)\cong\cok(h_1^\ast)\oplus\cok(e_1^\ast)$.
We see that the module $\cok(h_1^\ast)$ is isomorphic to $\tr(\tr M)$.
Note that $G_0=F_1^\ast$ and $H_0$ are free modules of the same rank.
Hence $e_1$ is a map from $E_1$ to $E_0=0$.
Therefore, the module $\cok(e_1^\ast)$ is isomorphic to $E_1^\ast$, which is a free module with
$$
\rank E_1^\ast=\rank E_1=\rank G_1-\rank H_1=\rank F_0^\ast-\rank H_1=\rank F_0-\rank H_1=\beta_0(M)-\beta_1(\tr M).
$$
Consequently, we obtain a desired isomorphism $M\cong\tr(\tr M)\oplus R^{\oplus\beta_0(M)-\beta_1(\tr M)}$ of $R$-modules.
\end{proof}

Next we recall some notation about subcategories of the module category.
The symbols $[-]$ and $[-]_n$ are introduced in \cite{radius} to define the radius of each subcategory of the module category.

\begin{dfn}
\begin{enumerate}[(1)]
\item
We denote by $\mod R$ the category of (finitely generated) $R$-modules.
\item
For an $R$-module $X$ we denote by $\add X$ the {\em additive closure} of $X$, which is defined to be the subcategory of $\mod R$ consisting of all direct summands of finite direct sums of copies of $X$.
\item
For a subcategory $\X$ of $\mod R$, we let $[\X]$ stand for the smallest subcategory of $\mod R$ that contains $\X$ and is closed under finite direct sums, direct summands and syzygies.
\item
For subcategories $\X,\Y$ of $\mod R$, denote by $\X\circ\Y$ the subcategory of $\mod R$ consisting of modules $M$ such that there is an exact sequence $0\to X\to M\to Y\to0$ with $X\in\X$ and $Y\in\Y$.
We set $\X\bullet\Y=[[\X]\circ[\Y]]$.
\item
For a subcategory $\X$ of $\mod R$, we set $[\X]_0=0$ and $[\X]_n=[\X]_{n-1}\bullet\X=[[\X]_{n-1}\circ[\X]]$ for each $n\ge1$.
\end{enumerate}
\end{dfn}

\begin{rem}
Let $M,N\in\mod R$ and $n>0$.
Then $N\in[M]_n$ if and only if there exists an exact sequence $0\to A\to B\to C\to0$ of $R$-modules with $A\in[M]_{n-1}$ and $C\in[M]_1$ such that $N$ is a direct summand of $B$; see \cite[Proposition 2.2]{radius}.
Using this, we easily observe that if $N\in[M]_n$, then $\syz^iN\in[\syz^iM]_n$ for all $i\ge0$.
\end{rem}

We also need to use the notion of $n$-torsionfree modules and the subcategory $\G_{m,n}$ of the module category, which are introduced and studied by Auslander and Bridger \cite{AB} and Iyama \cite{I}.

\begin{dfn}
For nonnegative integers $m$ and $n$, we denote by $\G_{n,m}$ the subcategory of $\mod R$ consisting of those $R$-modules $M$ which satisfy the equalities $\Ext_R^i(M,R)=\Ext_R^j(\tr M,R)=0$ for all integers $1\le i\le n$ and $1\le j\le m$.
An {\em $n$-torsionfree} $R$-module is defined as an $R$-module that belongs to the subcategory $\G_{0,n}$.
\end{dfn}

We prepare a lemma which is used to give a proof of the main result of this section.

\begin{lem}\label{14}
Let $n,s$ be nonnegative integers with $n\ge s$.
Let $M$ be an $n$-torsionfree $R$-module, and $N$ an $R$-module with $\pd N\le s$.
Then $\syz^s\Hom(\tr M,N)\in[\syz^2M]_{s+1}$ and $\Ext^i(\tr M,N)=0$ for every $1\le i\le n-s$.
\end{lem}

\begin{proof}
There is an exact sequence $0\to R^{\oplus b_s}\to\cdots\to R^{\oplus b_0}\to N\to0$ of $R$-modules with $b_i\in\N$ for each $i$.
We show the lemma by induction on $s$.
Let $s=0$.
Then $N$ is isomorphic to $R^{\oplus b_0}$.
By Lemma \ref{1}(2) we have
$$
\begin{array}{l}
\Hom(\tr M,N)\cong\Hom(\tr M,R^{\oplus b_0})\cong\Hom(\tr M,R)^{\oplus b_0}\cong(\syz^2M)^{\oplus b_0}\in[\syz^2M]_1,\text{ and}\\
\Ext^i(\tr M,N)\cong\Ext^i(\tr M,R^{\oplus b_0})\cong\Ext^i(\tr M,R)^{\oplus b_0}=0\text{ for all integers }1\le i\le n.
\end{array}
$$
Let $s>0$.
Take the kernel $L$ of the map $R^{\oplus b_0}\to N$.
Then $\pd L\le s-1$ and we get an exact sequence $\sigma:0\to L\to R^{\oplus b_0}\to N\to0$.
As $n\ge s-1\ge0$, we can apply the induction hypothesis to get $\syz^{s-1}\Hom(\tr M,L)\in[\syz^2M]_s$ and $\Ext^i(\tr M,L)=0$ for each $1\le i\le n-s+1$.
Since $n-s+1\ge1$, we have $\Ext^1(\tr M,L)=0$.
Applying $\Hom(\tr M,-)$ to $\sigma$ yields an exact sequence $0\to\Hom(\tr M,L)\to(\syz^2M)^{\oplus b_0}\to\Hom(\tr M,N)\to0$ and an equality $\Ext^i(\tr M,N)=0$ for all $1\le i\le n-s$.
Using Lemma \ref{15}(1), we obtain an exact sequence $0\to(\syz^{s+2}M)^{\oplus b_0}\to\syz^s\Hom(\tr M,N)\oplus F\to\syz^{s-1}\Hom(\tr M,L)\to0$ with $F$ free.
Since $(\syz^{s+2}M)^{\oplus b_0}$ and $\syz^{s-1}\Hom(\tr M,L)$ are in $[\syz^2M]_1$ and $[\syz^2M]_s$ respectively, $\syz^s\Hom(\tr M,N)$ belongs to $[\syz^2M]_{s+1}$.
\end{proof}

To state our theorem we need to recall the definition of a cosyzygy of a module.

\begin{dfn}
Let $M$ be an $R$-module.
Let $\lambda:M\to M^{\ast\ast}$ be the natural homomorphism.
Let $\cdots\to F_1\to F_0\xrightarrow{\pi}M^\ast\to0$ be a minimal free resolution of $M^\ast$.
The {\em first cosyzygy} $\syz_R^{-1}M$ of $M$ is defined as the cokernel of the composite map $\pi^\ast\lambda:M\to F_0^\ast$.
Note that if $f_1,\dots,f_r$ is a minimal system of generators of $M^\ast$, then the map $\pi^\ast\lambda$ is identified with the map $M\to R^{\oplus r}$ given by the transpose of the matrix $(f_1,\dots,f_r)$.
For each integer $n\ge2$, we define the {\em $n$th cosyzygy} $\syz_R^{-n}M$ inductively by $\syz_R^{-n}M=\syz_R^{-1}(\syz_R^{-(n-1)}M)$.
The $n$th cosyzygy of $M$ is uniquely determined by $M$ up to isomorphism, since so is a minimal free resolution of $M^\ast$.
Note by definition that one has $\syz_R^{-i}P=0$ for every free $R$-module $P$ and every positive integer $i$.
\end{dfn}

Now we can state and prove the following theorem, which is the main result of this section.
In fact, this theorem plays a fundamental role to obtain all the main results of this paper that are stated later.

\begin{thm}\label{19}
Let $1\le n\le t+1$.
Let $M,N$ be $R$-modules.
Assume that $\syz^nN$ belongs to $\add(R\oplus\syz^{n+1}k)$.
\begin{enumerate}[\rm(1)]
\item
Suppose that $M$ is $n$-torsionfree.
Then the $R$-module $\Ext^1(\tr M,\syz^{-n}\syz^nN)$ is a $k$-vector space, and
$$
\syz^{n-1}\Hom(\tr M,\syz^{-n}\syz^nN)\in[\syz^2M]_{n+1}\subseteq[M]_{n+1},\quad
\syz^n\Hom(\tr M,N)\in[\syz^2M]_{2n+1}\subseteq[M]_{2n+1}.
$$
\item
Suppose that $M$ is $(n+1)$-torsionfree.
Then the $R$-module $\Ext^1(\tr M,N)$ is a $k$-vector space, and
$$
\syz^{n+2}\Ext^1(\tr M,N)\in[\syz^2M\oplus\syz^{n+1}N]_{4n+3}\subseteq[M\oplus N]_{4n+3}.
$$
\end{enumerate}
\end{thm}

\begin{proof}
By assumption, $\syz^nN$ is a direct summand of $P\oplus\syz^{n+1}k^{\oplus a}$ for some free $R$-module $P$ and some integer $a\ge0$.
Note that any $k$-vector space has grade at least $t$ as an $R$-module.
The inequality $\grade\Ext^i(\m,R)\ge i-1$ holds for each integer $1\le i\le n$, since $\Ext^i(\m,R)=\Ext^{i+1}(k,R)$ is a $k$-vector space and $i-1\le n-1\le t$.
By \cite[Proposition (2.26)]{AB}, the $R$-module $\syz^{n+1}k=\syz^n\m$ is $n$-torsionfree, so is $P\oplus\syz^{n+1}k^{\oplus a}$, and so is $\syz^nN$.
It follows from \cite[Proposition (2.21)]{AB} that there exist two exact sequences of $R$-modules:
$$
\sigma:\ 0\to X\to F\oplus\syz^{-n}\syz^{n+1}k\to\m\to0,\qquad
\tau:\ 0\to Y\to G\oplus\syz^{-n}\syz^nN\to N\to0
$$
with $X,Y$ of projective dimension less than $n$ and $F,G$ free.
Set $H=F\oplus\syz^{-n}\syz^{n+1}k$ and $K=G\oplus\syz^{-n}\syz^nN$.

(1) Let $M$ be an $n$-torsionfree $R$-module.
Using Lemma \ref{14} for $s=n-1$, we see that $\syz^{n-1}\Hom(\tr M,Z)$ belongs to $[\syz^2M]_n$ and $\Ext^1(\tr M,Z)=0$ for each $Z\in\{X,Y\}$.
Lemma \ref{1}(2) shows that $\Hom(\tr M,\m)$ is isomorphic to $\syz^2M$.
Applying the functor $\Hom(\tr M,-)$ to $\sigma$ and $\tau$, we get exact sequences of $R$-modules:
$$
\begin{array}{l}
0\to\Hom(\tr M,X)\to\Hom(\tr M,H)\to\syz^2M\to0,\\
0\to\Hom(\tr M,Y)\to\Hom(\tr M,K)\to\Hom(\tr M,N)\to0.
\end{array}
$$
Using Lemma \ref{15}(1) for these exact sequences, we obtain exact sequences of $R$-modules with $Q,U$ free:
$$
\begin{array}{l}
\zeta:\ 0\to\syz^{n-1}\Hom(\tr M,X)\to\syz^{n-1}\Hom(\tr M,H)\oplus Q\to\syz^{n+1}M\to0,\\
\eta:\ 0\to\syz^n\Hom(\tr M,K)\to\syz^n\Hom(\tr M,N)\oplus U\to\syz^{n-1}\Hom(\tr M,Y)\to0.
\end{array}
$$
The modules $\syz^{n-1}\Hom(\tr M,X)$ and $\syz^{n+1}M$ are in $[\syz^2M]_n$ and $[\syz^2M]_1$, respectively.
The exact sequence $\zeta$ shows that $\syz^{n-1}\Hom(\tr M,H)$ is in $[\syz^2M]_{n+1}$, and so is its direct summand $\syz^{n-1}\Hom(\tr M,\syz^{-n}\syz^{n+1}k)$.
Since $\syz^nN$ is a direct summand of $P\oplus\syz^{n+1}k^{\oplus a}$, the module $\syz^{n-1}\Hom(\tr M,\syz^{-n}\syz^nN)$ is a direct summand of $\syz^{n-1}\Hom(\tr M,\syz^{-n}\syz^{n+1}k)^{\oplus a}$ as $\syz^{-n}P=0$.
Therefore, the module $\syz^{n-1}\Hom(\tr M,\syz^{-n}\syz^nN)$ belongs to $[\syz^2M]_{n+1}$.
Lemma \ref{1}(2) shows $\syz^{n-1}\Hom(\tr M,R^{\oplus r})\cong\syz^{n-1}\syz^2M^{\oplus r}\in[\syz^2M]_1$ for any $r\in\N$.
Hence
$$
\syz^{n-1}\Hom(\tr M,K)=\syz^{n-1}\Hom(\tr M,G)\oplus\syz^{n-1}\Hom(\tr M,\syz^{-n}\syz^nN)\in[\syz^2M]_{n+1},
$$
and thus $\syz^n\Hom(\tr M,K)$ is in $[\syz^2M]_{n+1}$.
As $\syz^{n-1}\Hom(\tr M,Y)$ is in $[\syz^2M]_n$, the exact sequence $\eta$ shows that $\syz^n\Hom(\tr M,N)$ belongs to $[\syz^2M]_{2n+1}$.
The $R$-module $\Ext^1(\tr M,\syz^{-n}\syz^nN)$ is a direct summand of
$$
E:=\Ext^1(\tr M,\syz^{-n}(P\oplus\syz^{n+1}k^{\oplus a}))=\Ext^1(\tr M,\syz^{-n}\syz^{n+1}k)^{\oplus a}.
$$
As $M$ is $n$-torsionfree and $n$ is positive, $\Ext^1(\tr M,R)$ vanishes.
Since $\Ext^1(\tr M,X)=0$, the exact sequence $\sigma$ induces a monomorphism $E\hookrightarrow\Ext^1(\tr M,\m)^{\oplus a}$.
Lemma \ref{1}(2) shows that $\Ext^1(\tr M,\m)$ is isomorphic to the $k$-vector space $\Hom(\tr M,k)$.
Consequently, the module $\Ext^1(\tr M,\syz^{-n}\syz^nN)$ is a $k$-vector space.

(2) Let $M$ be $(n+1)$-torsionfree.
Lemma \ref{14} implies $\Ext^i(\tr M,Y)=0$ for $1\le i\le(n+1)-(n-1)=2$.
From $\tau$ we get an isomorphism $\Ext^1(\tr M,G\oplus\syz^{-n}\syz^nN)\cong\Ext^1(\tr M,N)$.
Since $\Ext^1(\tr M,R)$ vanishes, $\Ext^1(\tr M,N)$ is isomorphic to $\Ext^1(\tr M,\syz^{-n}\syz^nN)$, which is a $k$-vector space by (1).
An exact sequence $0\to\syz\tr M\to R^{\oplus b}\to\tr M\to0$ with $b\in\N$ induces exact sequences $0\to\Hom(\tr M,N)\to N^{\oplus b}\to C\to0$ and $0\to C\to\Hom(\syz\tr M,N)\to\Ext^1(\tr M,N)\to0$.
Applying Lemma \ref{15}(1), we obtain exact sequences
$$
\begin{array}{l}
\lambda:\ 0\to\syz^{n+1}N^{\oplus b}\to \syz^{n+1}C\oplus R^{\oplus c}\to\syz^n\Hom(\tr M,N)\to0,\\
\rho:\ 0\to\syz^{n+2}\Hom(\syz\tr M,N)\to\syz^{n+2}\Ext^1(\tr M,N)\oplus R^{\oplus d}\to\syz^{n+1}C\to0.
\end{array}
$$
By (1) the module $\syz^n\Hom(\tr M,N)$ belongs to $[\syz^2M]_{2n+1}$ , while the module $\syz^{n+1}N^{\oplus b}$ is in $[\syz^{n+1}N]_1$.
The exact sequence $\lambda$ shows that $\syz^{n+1}C$ belongs to $[\syz^2M\oplus\syz^{n+1}N]_{2n+2}$.
Now, set $M'=\tr\syz\tr M$.
Applying \cite[Proposition 1.1.1]{I} to the $(n+1)$-torsionfree module $M$, we observe that $M'$ is in $\G_{1n}$, and in particular, $M'$ is $n$-torsionfree.
Thus we can apply (1) to $M'$ to get the containment $\syz^n\Hom(\tr M',N)\in[\syz^2M']_{2n+1}$.
Lemma \ref{15}(2) gives an isomorphism $\tr M'\oplus R^{\oplus e}\cong\syz\tr M$ with $e\in\N$.
As $M$ is $1$-torsionfree, applying \cite[Proposition 1.1.1]{I} again, we see that $\syz M'=\syz\tr\syz\tr M$ is isomorphic to $M$ up to free summands.
We have
$$
\syz^n\Hom(\syz\tr M,N)\cong\syz^n\Hom(\tr M',N)\oplus\syz^nN^{\oplus e}\in[\syz^2M'\oplus\syz^nN]_{2n+1}=[\syz M\oplus\syz^nN]_{2n+1}.
$$
Taking the second syzygies, we see that $\syz^{n+2}\Hom(\syz\tr M,N)$ is in $[\syz^3M\oplus\syz^{n+2}N]_{2n+1}$, which is contained in $[\syz^2M\oplus\syz^{n+1}N]_{2n+1}$.
The exact sequence $\rho$ shows that $\syz^{n+2}\Ext^1(\tr M,N)$ is in $[\syz^2M\oplus\syz^{n+1}N]_{4n+3}$.
\end{proof}

\section{Syzygies of the residue field}

In this section, applying the general theorem obtained in the previous section, we prove a theorem regarding the structure of syzygies of the residue field. 
First of all, we recall a couple of definitions, including those of a Burch ring and a local ring with quasi-decomposable maximal ideal.
These two notions of local rings are introduced and investigated in \cite{burch} and \cite{fiber}, respectively.
Studies of these two notions have been proceeded by many people and various results have been obtained so far; see \cite{AINS,CK,DE,DM,DK,restf,GS,qf,gorfib,R,dlr}.

\begin{dfn}
\begin{enumerate}[(1)]
\item
We denote by $\edim R$ the {\em embedding dimension} of $R$, that is, $\edim R=\nu_R(\m)$.
\item
We say that $R$ is a {\em hypersurface} if the inequality $\edim R-\depth R\le1$ holds.
This is equivalent to saying that $\widehat R\cong S/(f)$ for some regular local ring $(S,\n)$ and some element $f$ in $\n$; see \cite[the beginning of \S5.1]{A}.
\item
We say that $R$ is a {\em Burch ring} provided that there exist a maximal $\widehat R$-regular sequence $\xx$ in $\widehat R$, a regular local ring $(S,\n)$ and an ideal $I$ of $S$ such that $\n(I:_S\n)\ne\n I$ (i.e., $I$ is a {\em Burch ideal} of $S$) and $\widehat R/(\xx)\cong S/I$.
\item
We say that $\m$ is {\em quasi-decomposable} if $\m/(\xx)$ is decomposable over $R$ for some $R$-regular sequence $\xx$.
\end{enumerate}
\end{dfn}

We state several examples of Burch rings and local rings with quasi-decomposable maximal ideal.

\begin{ex}\label{37}
\begin{enumerate}[(1)]
\item
A local hypersurface is a Burch ring.
More generally, a local ring is a Gorenstein Burch ring if and only if it is a hypersurface; we refer the reader to \cite[Proposition 5.1]{burch}.
\item
A Cohen--Macaulay local ring with minimal multiplicity and infinite residue field is Burch by \cite[Proposition 5.2]{burch}.
If it is not a hypersurface, its maximal ideal is quasi-decomposable by \cite[Example 4.7]{fiber}.
\item
A local ring whose completion has a rational surface singularity is Burch, and has quasi-decomposable maximal ideal if it is not a hypersurface.
This is by (2) and \cite[Proposition 3.8]{ST} (cf. \cite
[Theorem 4]{Ar}).
\item
The fiber product $A\times_kB$ of two local rings $A$ and $B$ with common residue field $k$ has decomposable maximal ideal, and hence it is quasi-decomposable.
This is none other than Lemma \ref{6}(1).
\item
A local ring which deforms to an $(\mathrm{A}_1)$-singularity of dimension one is a local hypersurface with minimal multiplicity and quasi-decomposable maximal ideal.
In particular, a regular local ring of dimension at least two has quasi-decomposable maximal ideal.
Indeed, let $(R,\m,k)$ be a local ring of positive dimension $d$.
Suppose that there exist an $R$-regular sequence $\xx=x_1,\dots,x_{d-1}$, a regular local ring $S$ of dimension two, and a regular system of parameters $y,z$ of $S$ such that $R/(\xx)$ is isomorphic to $S/(yz)$.
Then $R$ is a Cohen--Macaulay ring, and one has $\m/\m^2\cong((\m/(\xx))/(\m/(\xx))^2)\oplus(\m^2+(\xx)/\m^2)$.
Hence it holds that
$$
\edim R-\depth R=\edim R/(\xx)+\dim_k(\m^2+(\xx)/\m^2)-d\le2+(d-1)-d=1.
$$
Therefore, the local ring $R$ is a hypersurface.
Using \cite[Theorem 14.9]{M}, we have that $\e(R)\le\e(R/(\xx))=\e(S/(yz))=2$, where for a local ring $A$ we denote by $\e(A)$ the (Hilbert--Samuel) multiplicity of $A$.
It is observed that the equality $\e(R)=\edim R-\dim R+1$ holds, which means that $R$ has minimal multiplicity.
\end{enumerate}
\end{ex}

The $(t+2)$nd syzygy of $k$ satisfies a remarkable property provided $R$ is Burch or $\m$ is quasi-decomposable.

\begin{prop}\label{10}
\begin{enumerate}[\rm(1)]
\item
If the local ring $R$ is a singular hypersurface, then $\syz^tk$ is isomorphic to $\syz^{t+2}k$.
\item
If the local ring $R$ is a singular Burch ring, then $\syz^tk$ is a direct summand of $\syz^{t+2}k$.
\item
If the maximal ideal $\m$ of $R$ is quasi-decomposable, then $\syz^{t+1}k$ is a direct summand of $\syz^{t+2}k$.
\end{enumerate}
\end{prop}

\begin{proof}
The first assertion follows from \cite[the proof of Theorem 5.1.1 and Construction 5.1.2]{A}, while the second is shown in \cite[Proposition 5.10]{burch}.
We prove the third assertion.
If $R$ is regular, then $\syz^{t+1}k=\syz^{t+2}k=0$ and the assertion holds.
So, we assume that $R$ is singular.
As $\m$ is quasi-decomposable, there exists an $R$-sequence $\xx=x_1,\dots,x_n$ such that $\m/(\xx)$ is decomposable.
By Corollary \ref{7}, we see that $\syz_{R/(\xx)}k=\m/(\xx)$ is a direct summand of $\syz_{R/(\xx)}^2k$.
Hence $\syz_R^t\syz_{R/(\xx)}k$ is a direct summand of $\syz_R^t\syz_{R/(\xx)}^2k$.
Note that $t\ge n$.
In view of \cite[Lemma 4.2]{fiber}, there are free $R$-modules $F,G$ such that $\syz_R^t\syz_{R/(\xx)}k\cong\syz_R^{t+1}k\oplus F$ and $\syz_R^t\syz_{R/(\xx)}^2k\cong\syz_R^{t+2}k\oplus G$.
Thus $\syz_R^{t+1}k$ is a direct summand of $\syz_R^{t+2}k\oplus G$.
Since $R$ is singular, $\syz_R^{t+1}k$ has no nonzero free summand by \cite[Proposition 7]{Ma}.
By \cite[Corollaries 1.10 and 1.15(i)]{LW}, we observe that $\syz_R^{t+1}k$ is a direct summand of $\syz_R^{t+2}k$.
\end{proof}

\begin{rem}
By \cite[Theorem 4.1]{burch}, if $k$ is a direct summand of $\syz^2k$, then $R$ is Burch.
Hence, the converse of Proposition \ref{10}(2) holds if $t=0$.
In view of \cite[Question 5.11]{burch}, we do not know if this holds even for $t>0$.
On the other hand, Example \ref{44} shows that the converse of Proposition \ref{10}(3) does not hold in general.
\end{rem}

We denote by $\mod_0R$ the subcategory of $\mod R$ consisting of those $R$-modules which are locally free on the punctured spectrum of $R$.
For an $R$-module $M$, we denote by $\nf(M)$ the {\em nonfree locus} of $M$, which is by definition the set of prime ideals $\p$ of $R$ such that the localization $M_\p$ is nonfree over the local ring $R_\p$.
The lemma below is necessary when we state a proof of the main result of this section.

\begin{lem}\label{21}
Let $M$ be an $R$-module.
The following three statements hold true.
\begin{enumerate}[\rm(1)]
\item
Let $X$ be an $R$-module with $X\in[M]_n$ for some integer $n\ge0$.
Then $\depth X\ge\inf\{\depth M,\depth R\}$.
\item
Put $e=\edim R$.
There exists an $R$-module $N\in[M]_{2^e}\cap\mod_0R$ such that $\pd_RN=\pd_RM$.
\item
If $I$ is an ideal of $R$ with $IM=0$, then there is an exact sequence $0\to I^{\oplus\nu_R(M)}\to\syz_RM\to\syz_{R/I}M\to0$.
\end{enumerate}
\end{lem}

\begin{proof}
(1) We observe by using \cite[Proposition 1.12(2)]{stcm} that the given inequality is satisfied.

(2) If $M$ is free, we can take $N:=M$.
Assume that $M$ is nonfree.
Choose a system of generators $x_1,\dots,x_e$ of the maximal ideal $\m$ of $R$.
Set $M_0=M$.
By \cite[Lemma 7.2]{burch}, for any $1\le i\le e$ there is an exact sequence $0\to\syz M_{i-1}\to M_i\to M_{i-1}\to0$ such that $\pd M_i\ge\pd M_{i-1}$ and $\nf(M_i)\subseteq\V(x_i)$.
Hence $\pd M_e\ge\pd M$.
The exact sequence shows $M_i\in[M_{i-1}]_2$ and $\nf(M_i)\subseteq\nf(M_{i-1})$ for each $1\le i\le e$.
It is observed that $M_e\in[M]_{2^e}$ and $\nf(M_e)\subseteq\V(x_1,\dots,x_e)=\{\m\}$.
Setting $N=M_e$, we have that $N\in[M]_{2^e}\cap\mod_0R$ and $\pd N\ge\pd M$.
Using (1) and the Auslander--Buchsbaum formula, we get $\pd N\le\pd M$.
Thus $\pd N=\pd M$.

(3) Put $n=\nu_R(M)=\nu_{R/I}(M)$.
There exists a commutative diagram with exact rows
$$
\xymatrix@R-1pc@C5pc{
0\ar[r]& \syz_RM\ar[r]\ar@{.>}[d]& R^{\oplus n}\ar[r]\ar[d]^-{\pi^{\oplus n}}& M\ar[r]\ar@{=}[d]& 0\\
0\ar[r]& \syz_{R/I}M\ar[r]& (R/I)^{\oplus n}\ar[r]& M\ar[r]& 0
}
$$
where $\pi:R\to R/I$ is the natural surjection.
The snake lemma gives rise to a desired exact sequence.
\end{proof}

\begin{rem}
It follows from \cite[Theorem 4.3]{res} that for every nonfree $R$-module $M$ there exists an $R$-module $N$ which is locally free on the punctured spectrum of $R$ and which belongs to the resolving closure of $M$.
We should notice that Lemma \ref{21}(2) provides a refinement of this statement.
\end{rem}

We denote by $\r(R)$ the {\em type} of $R$ and by $\soc R$ the {\em socle} of $R$, respectively, that is, $\r(R)=\dim_k\Ext_R^t(k,R)$ and $\soc R=0:_R\m\cong k^\ast$.
Now we can state and prove the following theorem, which is the main result of this section.
This theorem plays a crucial role in getting an upper bound of the Orlov spectrum of the singularity category in the next section.
In the proof of the theorem it is essential to invoke Theorem \ref{19}.

\begin{thm}\label{9}
Let the local ring $R$ have embedding dimension $e$.
Put $s=1$ when $t=0$, and $s=2^e$ when $t>0$.
Then the following two implications hold true for every $R$-module $M$ of infinite projective dimension.
\begin{align*}
\syz^{t+1}k\in\add(R\oplus\syz^{t+2}k)&\implies\syz^{t+1}k\in[\syz^{t+3}M]_{s(2t+3)}\subseteq[M]_{s(2t+3)},\\
\syz^tk\in\add(R\oplus\syz^{t+2}k)&\implies\syz^{t+1}k\in[\syz^{t+3}M]_{s(2t+4)}\subseteq[M]_{s(2t+4)}.
\end{align*}
\end{thm}

\begin{proof}
Lemma \ref{21}(2) gives rise to an $R$-module $K\in[M]_{2^e}\cap\mod_0R$ such that $\pd K=\pd M=\infty$.
As $K$ is locally free on the punctured spectrum, $\Ext^i(K,R)$ has finite length for all $i>0$.
Hence $\grade\Ext^i(K,R)\ge t\ge i-1$ for all $1\le i\le t+1$.
It follows from \cite[Proposition (2.26)]{AB} that $\syz^{t+1}K$ is $(t+1)$-torsionfree.
As $\syz M$ is $1$-torsionfree, $\syz^{t+1}M$ is $(t+1)$-torsionfree when $t=0$.
Setting $L=M$ when $t=0$ and $L=K$ when $t>0$, we have that $L$ is an $R$-module of infinite projective dimension such that $\syz^{t+1}L$ is $(t+1)$-torsionfree.

Let us show the first implication in the theorem.
Applying Theorem \ref{19}(1) to $n:=t+1$ and $N:=k$, we get $\syz^{t+1}\Hom(\tr\syz^{t+1}L,k)\in[\syz^2\syz^{t+1}L]_{2(t+1)+1}=[\syz^{t+3}L]_{2t+3}$.
If $\Hom(\tr\syz^{t+1}L,k)$ is zero, then so is $\tr\syz^{t+1}L$, and $\syz^{t+1}L$ is free by Lemma \ref{1}(4), which implies that $\pd L<\infty$, a contradiction.
Thus $\Hom(\tr\syz^{t+1}L,k)$ is a nonzero $k$-vector space, and $\syz^{t+1}k$ belongs to $[\syz^{t+3}L]_{2t+3}$.
When $t=0$, we have $\syz^{t+1}k\in[\syz^{t+3}L]_{2t+3}=[\syz^{t+3}M]_{2t+3}$.
When $t>0$, we have $\syz^{t+1}k\in[\syz^{t+3}L]_{2t+3}=[\syz^{t+3}K]_{2t+3}\subseteq[\syz^{t+3}M]_{2^e(2t+3)}$.

From here to the end of the proof, we prove the second implication in the theorem holds.
Set $X=\syz^{t+1}L$ and $E=\Ext^1(X,\syz X)$.
We claim that there exists an $R$-regular sequence $\xx=x_1,\dots,x_t$ which annihilates the $R$-module $E$.
In fact, there is nothing to show when $t=0$.
Let $t>0$.
Then $L=K$ is locally free on the punctured spectrum of $R$, which implies that the $R$-module $E$ has finite length.
Thus the claim follows.

Put $\overline{()}=()\otimes_RR/(\xx)$.
It follows from \cite[Lemma 5.1]{fiber} that $\xx$ is an $X$-regular sequence.
We have that
\begin{equation}\label{23}
\textstyle\syz^t\overline X=\syz^t(X/\xx X)\cong\bigoplus_{i=0}^t(\syz^iX)^{\oplus\binom{t}{i}}\in[X]_1
\end{equation}
by \cite[Corollary 3.2(1)]{kos} (or \cite[Proposition 2.2]{stcm}) and \cite[Lemma 2.14]{ua} when $t>0$.
This holds true even when $t=0$; indeed, $\syz^t\overline X=\syz^0X=X\in[X]_1$.
There exists an exact sequence $0\to k^{\oplus r}\to\overline R\to\overline R/\soc\overline R\to0$ of $\overline R$-modules, where $r=\r(R)>0$.
Applying Lemma \ref{15}(1), we get an exact sequence $0\to(\syz^tk)^{\oplus r}\to P\to \syz^t(\overline R/\soc\overline R)\to0$ of $R$-modules with $P$ free.
Hence $(\syz^tk)^{\oplus r}\cong\syz^{t+1}(\overline R/\soc\overline R)\oplus Q$ for some free $R$-module $Q$.
As $\syz^tk$ belongs to $\add(R\oplus\syz^{t+2}k)$, so does the module $\syz^{t+1}(\overline R/\soc\overline R)$.
Applying Theorem \ref{19}(1), we observe that the module $\syz^{t+1}\Hom(\tr X,\overline R/\soc\overline R)$ belongs to $[\syz^2X]_{2t+3}$.
Lemma \ref{1}(3) yields isomorphisms
$$
\overline{\tr X}\cong\tr_{\overline R}\overline X\oplus\overline R^{\oplus a},\qquad
\tr_{\overline R}\overline X/(\soc\overline R)\tr_{\overline R}\overline X\cong\tr_{\overline R/\soc\overline R}(\overline X/(\soc\overline R)\overline X)\oplus(\overline R/\soc\overline R)^{\oplus b},
$$
where $a,b$ are nonnegative integers.
Since the sequence $\xx$ is $R$-regular, it is $\syz^tL$-regular by \cite[Lemma 5.1]{fiber}.
A minimal free resolution $F$ of the $R$-module $L$ induces an exact sequence $0\to X\xrightarrow{\rho}F_t\to\syz^tL\to0$, which induces an exact sequence $0\to\overline X\xrightarrow{\overline\rho}\overline{F_t}\to\overline{\syz^tL}\to0$.
The monomorphism $\overline\rho$ factors through the inclusion map $\m\overline{F_t}\hookrightarrow\overline{F_t}$.
It is observed that the ideal $\soc\overline R$ of $\overline R$ annihilates the $\overline R$-module $\overline X$.
There are isomorphisms
$$
\begin{array}{l}
\overline{\tr X}\otimes_{\overline R}(\overline R/\soc\overline R)
\cong(\tr_{\overline R}\overline X\oplus\overline R^{\oplus a})\otimes_{\overline R}(\overline R/\soc\overline R)
\cong\tr_{\overline R}\overline X/(\soc\overline R)\tr_{\overline R}\overline X\oplus(\overline R/\soc\overline R)^{\oplus a}\\
\phantom{\overline{\tr X}\otimes_{\overline R}(\overline R/\soc\overline R)}\cong\tr_{\overline R/\soc\overline R}(\overline X/(\soc\overline R)\overline X)\oplus(\overline R/\soc\overline R)^{\oplus(a+b)}
\cong\tr_{\overline R/\soc\overline R}\overline X\oplus(\overline R/\soc\overline R)^{\oplus(a+b)}.
\end{array}
$$
Making use of Lemma \ref{1}(2), we have a series of isomorphisms of $R$-modules:
$$
\begin{array}{l}
\syz^{t+1}\Hom(\tr X,\overline R/\soc\overline R)
\cong\syz^{t+1}\Hom_{\overline R}(\overline{\tr X},\overline R/\soc\overline R)\\
\phantom{\syz^{t+1}\Hom(\tr X,\overline R/\soc\overline R)}
\cong\syz^{t+1}\Hom_{\overline R/\soc\overline R}(\overline{\tr X}\otimes_{\overline R}(\overline R/\soc\overline R),\overline R/\soc\overline R)\\
\phantom{\syz^{t+1}\Hom(\tr X,\overline R/\soc\overline R)}\cong\syz^{t+1}\Hom_{\overline R/\soc\overline R}(\tr_{\overline R/\soc\overline R}\overline X\oplus(\overline R/\soc\overline R)^{\oplus(a+b)},\overline R/\soc\overline R)\\
\phantom{\syz^{t+1}\Hom(\tr X,\overline R/\soc\overline R)}\cong\syz^{t+1}\syz_{\overline R/\soc\overline R}^2\overline X\oplus\syz^{t+1}(\overline R/\soc\overline R)^{\oplus(a+b)}.
\end{array}
$$
Hence $\syz^{t+1}\syz_{\overline R/\soc\overline R}^2\overline X$ is a direct summand of $\syz^{t+1}\Hom(\tr X,\overline R/\soc\overline R)$, which belongs to $[\syz^2X]_{2t+3}$.
Thus
\begin{equation}\label{24}
\syz^{t+1}\syz_{\overline R/\soc\overline R}^2\overline X\in[\syz^2X]_{2t+3}.
\end{equation}
Put $c=\nu_{\overline R}(\overline X)$.
We have $c=\nu_R(X)=\beta_{t+1}^R(L)>0$ as $\pd L=\infty$.
From Lemma \ref{21}(3) and the isomorphism $\soc\overline R\cong k^{\oplus r}$, there is an exact sequence $\sigma:0\to k^{\oplus rc}\to\syz_{\overline R}\overline X\to\syz_{\overline R/\soc\overline R}\overline X\to0$.
The inclusion map $\syz_{\overline R}\overline X\hookrightarrow\overline R^{\oplus c}$ factors via $\m\overline R^{\oplus c}$, so that $\syz_{\overline R}\overline X$ is killed by $\soc\overline R$.
Thus $\sigma$ is an exact sequence of $(\overline R/\soc\overline R)$-modules.
Lemma \ref{15}(1) provides an exact sequence $0\to\syz_{\overline R/\soc\overline R}^2\overline X\to k^{\oplus rc}\oplus Y\to\syz_{\overline R}\overline X\to0$ of $(\overline R/\soc\overline R)$-modules with $Y$ free.
Using Lemma \ref{15}(1) again, we obtain an exact sequence of $R$-modules with $Z$ free:
\begin{equation}\label{22}
0\to\syz^{t+1}\syz_{\overline R/\soc\overline R}^2\overline X\to(\syz^{t+1}k)^{\oplus rc}\oplus \syz^{t+1}Y\oplus Z\to\syz^{t+1}\syz_{\overline R}\overline X\to0.
\end{equation}
By \cite[Lemma 4.2]{fiber} we have $\syz^t\syz_{\overline R}\overline X\cong\syz^{t+1}\overline X\oplus R^{\oplus u}$ with $u\in\N$, which implies $\syz^{t+1}\syz_{\overline R}\overline X\cong\syz^{t+2}\overline X$.
By \eqref{23}, $\syz^{t+2}\overline X$ is in $[\syz^2X]_1$.
It is seen from \eqref{24} and \eqref{22} that $\syz^{t+1}k\in[\syz^{t+3}L]_{2t+4}$.
When $t=0$, we have $L=M$ and $\syz^{t+1}k\in[\syz^{t+3}M]_{2t+4}$.
When $t>0$, we have $L=K\in[M]_{2^e}$ and $\syz^{t+1}k\in[\syz^{t+3}M]_{2^e(2t+4)}$.
\end{proof}

To show our next proposition, we establish a lemma which should be well-known.

\begin{lem}\label{18}
The equality $\depth\syz^nk=n$ holds for all integers $0\le n\le t$.
\end{lem}

\begin{proof}
The depth lemma yields $\depth\syz^nk\ge\inf\{n+\depth k,\depth R\}=\inf\{n,t\}=n$, so it suffices to show $\Ext^n(k,\syz^nk)\ne0$.
If $n=0$, then $\Ext^n(k,\syz^nk)\cong\Hom(k,k)=k\ne0$.
If $n>0$, then there is a nonsplit exact sequence $0\to\syz^nk\to F\to\syz^{n-1}k\to0$ with $F$ free, whence $\Ext^n(k,\syz^nk)\cong\Ext^1(\syz^{n-1}k,\syz^nk)\ne0$.
\end{proof}

The conditions $\syz^{t+1}k\in\add(R\oplus\syz^{t+2}k)$ and $\syz^tk\in\add(R\oplus\syz^{t+2}k)$ in Theorem \ref{9} might look too specific or artificial.
The following proposition explains that those two conditions are rather reasonable.

\begin{prop}
The following two conditions are equivalent to each other.
\begin{enumerate}[\rm(1)]
\item
There exist integers $1\le n\le t+1$ and $0\le m\le n$ such that $\syz^mk\in\add(R\oplus\syz^{n+1}k)$.
\item
One has either that $\syz^{t+1}k\in\add(R\oplus\syz^{t+2}k)$ or that $\syz^tk\in\add(R\oplus\syz^{t+2}k)$.
\end{enumerate}
\end{prop}

\begin{proof}
It is obvious that (2) implies (1).
Let us show that (1) implies (2). 
We first consider the case $m>t$.
Then $t+1\le m\le n\le t+1$, and $m=n=t+1$.
Hence $\syz^{t+1}k\in\add(R\oplus\syz^{t+2}k)$.
Next we consider the case $m\le t$.
There exists an integer $r\ge0$ such that $\syz^mk$ is a direct summand of $(R\oplus\syz^{n+1}k)^{\oplus r}$.
It holds that
$$
\begin{array}{l}
n\ge m=\depth\syz^mk\ge\depth(R\oplus\syz^{n+1}k)^{\oplus r}\ge\depth(R\oplus\syz^{n+1}k)\\
\phantom{n\ge m=\depth\syz^mk}=\inf\{\depth R,\depth\syz^{n+1}k\}\ge\inf\{t,\inf\{n+1,t\}\}=\inf\{t,n+1\}
\end{array}
$$
by Lemma \ref{18} and the depth lemma.
This forces us to have $t<n+1$, and $t+1\ge n\ge m\ge\inf\{t,n+1\}=t$.
Hence $(n,m)=(t+1,t+1),(t+1,t),(t,t)$.
We have $\syz^{t+1}k\in\add(R\oplus\syz^{t+2}k)$, or $\syz^tk\in\add(R\oplus\syz^{t+2}k)$, or $\syz^tk\in\add(R\oplus\syz^{t+1}k)$.
It remains to note that if $\syz^tk\in\add(R\oplus\syz^{t+1}k)$, then $\syz^{t+1}k\in\add(R\oplus\syz^{t+2}k)$.
\end{proof}

\section{Orlov spectra over uniformly dominant isolated singularities}

In this section, using the theorem obtained in the previous section, we shall investigate generation in the singularity category.
We start by recalling the definitions of several notions including the Orlov spectrum and the ultimate dimension of a triangulated category, which are the main targets of this section.

\begin{dfn}
Let $\T$ be a triangulated category.
\begin{enumerate}[(1)]
\item
Let $\X,\Y$ be subcategories of $\T$.
We denote by $\langle\X\rangle$ the smallest subcategory of $\T$ which contains $\X$ and which is closed under finite direct sums, direct summands and shifts.
We denote by $\X\ast\Y$ the subcategory of $\T$ which consists of objects $T\in\T$ such that there exists an exact triangle $X\to T\to Y\rightsquigarrow$ in $\T$ with $X\in\X$ and $Y\in\Y$.
We put $\X\diamond\Y=\langle\X\ast\Y\rangle$.
We set $\langle\X\rangle_0^\T=0$, and $\langle\X\rangle_n^\T=\langle\X\rangle_{n-1}^\T\diamond\langle\X\rangle^\T$ for $n\ge1$.
\item
Let $X,Y\in\T$.
Put $\level_\T^X(Y)=\inf\{n\in\ZZ_{\ge-1}\mid Y\in\langle X\rangle_{n+1}\}$, $\gt_\T(X)=\inf\{n\in\ZZ_{\ge-1}\mid\langle X\rangle_{n+1}=\T\}$, $
\ospec\T=\{\gt_\T(G)\mid G\in\T\text{ with }\gt_\T(G)<\infty\}$, $\dim\T=\inf(\ospec\T)$, and $\udim\T=\sup(\ospec\T)$.
These are called the {\em level} of $Y$ with respect to $X$, the {\em generation time} of $X$, the {\em Orlov spectrum} of $\T$, the {\em (Rouquier) dimension} of $\T$, and the {\em ultimate dimension} of $\T$, respectively.
\end{enumerate}
\end{dfn}

\begin{rem}
The definitions of a level and a generation time in this paper are slightly modified from the original ones.
To be more precise, $\level_\T^X(Y)$ is defined as $\inf\{n\in\N\mid Y\in\langle X\rangle_n\}$ in \cite{ABIM}, while it is defined as $\inf\{n\in\N\mid Y\in\langle X\rangle_{n+1}\}$ in \cite{BFK}.
Also, $\gt_\T(X)$ is defined to be $\inf\{n\in\N\mid\langle X\rangle_{n+1}=\T\}$ in \cite{BFK}.
Note that our definitions of a level and a generation time do coincide with the definitions given in \cite{BFK} whenever $\T\ne0$. 
\end{rem}

Now, after recalling the definition of a singularity category, we shall introduce the notions of a dominant index and a uniformly dominant local ring, which play a primary role in this section.

\begin{dfn}
We denote by $\db(R)$ the bounded derived category of $\mod R$.
Let $\ds(R)$ stand for the {\em singularity category} of $R$, which is defined as the Verdier quotient of $\db(R)$ by the perfect complexes.
We set
$$
\dx(R)=\inf\{n\in\ZZ_{\ge-1}\mid\text{$k\in\langle X\rangle_{n+1}^{\ds(R)}$ for every $0\ne X\in\ds(R)$}\}\in\ZZ_{\ge-1}\cup\{\infty\}
$$
and call this the {\em dominant index} of $R$.
We say that $R$ is {\em uniformly dominant} if $\dx(R)<\infty$.
\end{dfn}

\begin{rem}
\begin{enumerate}[(1)]
\item
If $R$ is a regular local ring, then $R$ is uniformly dominant with dominant index $-1$.
In fact, $\ds(R)=0$ and there exists no such object $0\ne X\in\ds(R)$ (or $k\cong0\in\langle X\rangle_0$ for all $X\in\ds(R)$).
\item
It is evident by definition that every uniformly dominant local ring is {\em dominant} in the sense of \cite{dlr}.
\end{enumerate}
\end{rem}

By virtue of Theorem \ref{9}, we obtain sufficient conditions for a local ring to be uniformly dominant.

\begin{cor}\label{40}
Let $e$ denote the embedding dimension of $R$.
Put $s=1$ when $t=0$, and $s=2^e$ when $t>0$.
\begin{enumerate}[\rm(1)]
\item
Suppose that $\syz^{t+1}k$ belongs to $\add(R\oplus\syz^{t+2}k)$; this condition is satisfied if the local ring $R$ has quasi-decomposable maximal ideal.
Then one has $\dx(R)\le s(2t+3)-1$.
In particular, $R$ is uniformly dominant.
\item
Suppose that $\syz^tk$ belongs to $\add(R\oplus\syz^{t+2}k)$; this condition is satisfied if $R$ is a singular Burch ring (e.g., if $R$ is a singular hypersurface).
Then $\dx(R)\le s(2t+4)-1$.
In particular, $R$ is uniformly dominant.
\end{enumerate}
\end{cor}

\begin{proof}
Fix a nonzero object $X$ of $\ds(R)$.
Then $X\cong M[r]$ in $\ds(R)$ for some $R$-module $M$ and some integer $r$; see \cite[Lemma 2.4(2a)]{sing} for instance.
Since $X$ is nonzero in $\ds(R)$, the $R$-module $M$ has infinite projective dimension.
Suppose that $\syz^{t+1}k$ belongs to $\add(R\oplus\syz^{t+2}k)$.
Then it follows from Theorem \ref{9} that $\syz^{t+1}k$ belongs to $[M]_{s(2t+3)}$.
This implies that in $\ds(R)$ the object $\syz^{t+1}k$ belongs to $\langle M\rangle_{s(2t+3)}=\langle X\rangle_{s(2t+3)}$.
Since $k\cong(\syz^{t+1}k)[t+1]$ in $\ds(R)$ by \cite[Lemma 2.4(2b)]{sing}, we see that $k$ belongs to $\langle X\rangle_{s(2t+3)}$.
The inequality $\dx(R)\le s(2t+3)-1$ is obtained.
In a similar way, we see that if $\syz^tk$ is in $\add(R\oplus\syz^{t+2}k)$, then the inequality $\dx(R)\le s(2t+4)-1$ holds.
The proof of the corollary is completed by invoking Proposition \ref{10}.
\end{proof}

The example below says that the converse of each assertion of Corollary \ref{40} is not true in general.

\begin{ex}\label{45}
We denote the length of a module $M$ over a ring $R$ by $\ell_R(M)$.
Let $k$ be a field.
\begin{enumerate}[(1)]
\item
The artinian local ring $R=k[x,y]/(x^3,x^2y,xy^2,y^3)$ is Burch by \cite[Corollary 6.5]{burch}, and so it is uniformly dominant by Corollary \ref{40}(2).
However, $\syz^{t+1}k$ is not in $\add(R\oplus\syz^{t+2}k)$.
In fact, assume that it does.
As $R$ is a complete local ring, $\m$ is a nonfree indecomposable $R$-module and $t=0$, we see by \cite[Corollaries 1.10, 1.15]{LW} and \cite[Proposition 7]{Ma} that $\m$ is a direct summand of $\syz^2k$.
There is an $R$-module $L$ with $\syz^2k\cong\m\oplus L$.
The exact sequence $0\to\syz^2k\to R^{\oplus2}\to\m\to0$ shows that $\ell_R(\syz^2k)=2\ell_R(R)-\ell_R(\m)=2\cdot6-5=7$, so that $\nu_R(L)\le\ell_R(L)=\ell_R(\syz^2k)-\ell_R(\m)=7-5=2$.
Therefore, $\beta_2^R(k)=\nu_R(\syz^2k)=\nu_R(\m)+\nu_R(L)\le2+2=4$.
This contradicts the fact that a minimal free resolution of $k$ has the form below, which shows $\beta_2^R(k)=5$.
$$
0\gets k\gets R\xleftarrow{(x,y)}R^{\oplus2}\xleftarrow{\left(\begin{smallmatrix}x^2&0&0&0&y\\0&x^2&xy&y^2&-x\end{smallmatrix}\right)}R^{\oplus5}\gets\cdots
$$
\item
The artinian local ring $R=k[x,y,z,w]/(x^2,y^2,z^2,w^2,xz,xw,yz,yw)$ is not Burch but has decomposable maximal ideal by \cite[Theorem 9.10]{dlr}.
It follows from Corollary \ref{40}(1) that the local ring $R$ is uniformly dominant.
However, we easily see from \cite[Theorem 4.1]{burch} that $\syz^tk$ does not belong to $\add(R\oplus\syz^{t+2}k)$.
\end{enumerate}
\end{ex}

We study generation of the singularity category by using a dominant index, a level and a generation time.

\begin{prop}\label{30}
Let $X$ be an object of $\ds(R)$.
The following two statements hold true.
\begin{enumerate}[\rm(1)]
\item
Assume that $n=\dx(R)$, $v=\level^k(X)$ and $g=\gt(X)$ are all finite.
Then $\ds(R)=\langle G\rangle_{(n+1)(v+1)(g+1)}$ for each $0\ne G\in\ds(R)$.
In particular, every nonzero object of $\ds(R)$ has finite generation time.
\item
There is an inequality $\udim\ds(R)\le(\dx(R)+1)(\level_{\ds(R)}^k(X)+1)(\gt_{\ds(R)}(X)+1)-1$.
\end{enumerate}
\end{prop}

\begin{proof}
(1) By definition, we have $X\in\langle k\rangle_{v+1}$ and $\langle X\rangle_{g+1}=\ds(R)$.
Since $G$ is nonzero, we get $k\in\langle G\rangle_{n+1}$.
It holds that $\ds(R)=\langle X\rangle_{g+1}\subseteq\langle k\rangle_{(v+1)(g+1)}\subseteq\langle G\rangle_{(n+1)(v+1)(g+1)}$.
We obtain $\ds(R)=\langle G\rangle_{(n+1)(v+1)(g+1)}$.

(2) If $R$ is regular, then $\ds(R)=0$, $\udim\ds(R)=-1$ and the inequality evidently holds.
We may assume that $R$ is singular, and that $n=\dx(R)$, $v=\level^k(X)$ and $g=\gt(X)$ are finite.
Let $G\in\ds(R)$ with $\gt(G)<\infty$.
Then $G\ne0$.
By (1) we get $\ds(R)=\langle G\rangle_{(n+1)(v+1)(g+1)}$.
Thus $\udim\ds(R)\le(n+1)(v+1)(g+1)-1$.
\end{proof}

As an application of the above proposition, we get information on upper and lower bounds of the Orlov spectrum of the singularity category of a uniformly dominant isolated singularity.

\begin{cor}
Assume that the local ring $R$ is uniformly dominant and has an isolated singularity.
If the singularity category $\ds(R)$ has finite Rouquier dimension, then it has finite ultimate dimension.
\end{cor}

\begin{proof}
By assumption, $g:=\dim\ds(R)$ is finite.
We find $X\in\ds(R)$ with $g=\gt(X)$.
Since $R$ is uniformly dominant, $n:=\dx(R)$ is finite.
As $R$ has an isolated singularity, $\ds(R)=\thick k$ by \cite[Corollary 4.3(2)]{kos}.
Hence $v:=\level^k(X)$ is finite.
Proposition \ref{30}(2) yields that $\udim\ds(R)\le(n+1)(v+1)(g+1)-1<\infty$.
\end{proof}

Here we need to recall the definition of the annihilator of the singularity category of the local ring $R$.

\begin{dfn}
We denote by $\ann_R\ds(R)$ the {\em annihilator} of $\ds(R)$, which is by definition the set of elements $a$ of $R$ such that the multiplication morphism $X\xrightarrow{a}X$ in $\ds(R)$ is zero for all $X\in\ds(R)$.
\end{dfn}

Under some mild assumptions, we can prove that the ultimate dimension of the singularity category of a uniformly dominant isolated singularity is finite, and actually we can get an explicit upper bound for it.

\begin{cor}\label{28}
Let $R$ be an excellent equicharacteristic uniformly dominant local ring which has an isolated singularity.
Let $J$ be an $\m$-primary ideal of $R$ which is contained in $\ann\ds(R)$.
Put $n=\dx(R)$, $m=\nu(J)$ and $l=\ell\ell(R/J)$.
Then one has $\ds(R)=\langle G\rangle_{(n+1)(m-t+1)l}$ for each $0\ne G\in\ds(R)$.
In particular, every nonzero object of $\ds(R)$ has finite generation time, and it holds that $\udim\ds(R)\le(n+1)(m-t+1)l-1$.
\end{cor}

\begin{proof}
It follows by \cite[Theorem 1.3(2)]{L} that $\ds(R)=\langle k\rangle_{(m-t+1)l}$, so that $\gt(k)\le(m-t+1)l-1$.
Thus
$$
(\dx(R)+1)(\level^k(k)+1)(\gt(k)+1)\le(n+1)(0+1)(((m-t+1)l-1)+1)=(n+1)(m-t+1)l.
$$
By virtue of (1) and (2) of Proposition \ref{30}, we observe that the equality $\ds(R)=\langle G\rangle_{(n+1)(m-t+1)l}$ holds true for all nonzero objects $G$ of $\ds(R)$, and that there is an inequality $\udim\ds(R)\le(n+1)(m-t+1)l-1$.
\end{proof}

We recall the definition of a Jacobian ideal of a complete equicharacteristic local ring.

\begin{dfn}
Suppose that the local ring $R$ is complete and equicharacteristic.
Then Cohen's structure theorem implies that $R$ is isomorphic to a quotient $k[\![x_1,\dots,x_n]\!]/(f_1,\dots,f_r)$ of a formal power series ring.
Put $h=\height(f_1,\dots,f_r)=n-\dim R$.
The {\em Jacobian ideal} of $R$, denoted $\jac R$, is defined as the ideal of $R$ generated by (the preimages in $R$ of) the $h\times h$ minors of the Jacobian matrix $\partial(f_1,\dots,f_r)/\partial(x_1,\dots,x_n)=(\partial f_i/\partial x_j)$.
\end{dfn}

The remark below provides examples of an ideal $J$ which satisfies the assumption given in Corollary \ref{28}.

\begin{rem}\label{29}
\begin{enumerate}[(1)]
\item
Let $R$ be an equicharacteristic excellent local ring with an isolated singularity.
Then the annihilator $\ann\ds(R)$ is an $\m$-primary ideal of $R$ (contained in $\ann\ds(R)$); see \cite[Theorem 1.3(1)]{L}.
\item
Let $R$ be an equicharacteristic complete Cohen--Macaulay local ring with perfect residue field, and assume that $R$ has an isolated singularity.
Then $\jac R$ is an $\m$-primary ideal of $R$ which is contained in $\ann\ds(R)$.
Indeed, it is observed from \cite[Lemmas 4.3 and Propositions 4.4, 4.5]{W} that the ideal $\jac R$ is $\m$-primary.
It follows from \cite[Theorem 5.3]{W} that $\jac R$ annihilates the $R$-module $\Ext^{d+1}(M,N)$ for all $R$-modules $M$ and $N$, where $d=\dim R$.
We thus see from \cite[Proposition 5.3(1)]{L} that $\jac R$ is contained in $\ann\ds(R)$.
\end{enumerate}
\end{rem}

We close the section by summarizing Corollaries \ref{40}, \ref{28} and Remark \ref{29}.

\begin{thm}\label{25}
Let $R$ be excellent, equicharacteristic and has an isolated singularity.
Let $J$ be an $\m$-primary ideal of $R$ contained in $\ann\ds(R)$ and set $m=\nu(J)$ and $l=\ell\ell(R/J)$; one can always take $J=\ann\ds(R)$, and one can even take $J=\jac R$ provided that $R$ is a complete Cohen--Macaulay local ring with $k$ perfect.
Put $e=\edim R$.
Let $s=1$ when $t=0$, and $s=2^e$ when $t>0$.
Then the following two statements hold true.
\begin{enumerate}[\rm(1)]
\item
Suppose that $\syz^{t+1}k$ belongs to $\add(R\oplus\syz^{t+2}k)$ (e.g., $\m$ is quasi-decomposable).
Then every nonzero object of $\ds(R)$ has finite generation time, and $\ds(R)$ has ultimate dimension at most $s(2t+3)(m-t+1)l-1$.
\item
Suppose that $\syz^tk$ belongs to $\add(R\oplus\syz^{t+2}k)$ (e.g., $R$ is a singular Burch ring).
Then every nonzero object of $\ds(R)$ has finite generation time, and $\ds(R)$ has ultimate dimension at most $s(2t+4)(m-t+1)l-1$.
\end{enumerate}
\end{thm}

\section{Producing uniformly dominant local rings}

In this section, we present several methods to produce a uniformly dominant local ring from a given one, following the methods given in \cite{dlr} to produce a dominant local ring from a given one.
We begin with stating a general basic lemma concerning an exact functor of Verdier quotients of triangulated categories.

\begin{lem}\label{4}
Let $F:\T\to\U$ be an exact functor of triangulated categories.
Let $\X$ and $\Y$ be thick subcategories of $\T$ and $\U$, respectively, with $F(\X)\subseteq\Y$.
Then there exists an exact functor $\overline F:\T/\X\to\U/\Y$ such that $\overline F(\alpha(T))=\beta(F(T))$ for each $T\in\T$, where $\alpha:\T\to\T/\X$ and $\beta:\U\to\U/\Y$ stand for the canonical functors.
\end{lem}

\begin{proof}
The composition $\beta F:\T\to\U/\Y$ is an exact functor and satisfies $(\beta F)(\X)\subseteq\beta(\Y)=0$.
Hence $\beta F$ factors via $\alpha$; see \cite[Theorem 2.1.8]{N}.
Thus we get a functor $\overline F:\T/\X\to\U/\Y$ as in the assertion follows.
\end{proof}

We can now show the theorem below, which is a uniformly dominant version of \cite[Theorem 5.6]{dlr}.
Let us explain notation used in the proof:
For an element $x$ of $R$ we denote by $\k(x,R)$ the Koszul complex of $x$ over $R$, i.e., $\k(x,R)=(0\to R\xrightarrow{x}R\to0)$.
By $(R/(x))_R$ we mean the (right) $R$-module $R/(x)$ via the surjective ring homomorphism $R\to R/(x)$.
For a complex $X$ of $R$-modules, we put $\inf X=\inf\{i\in\ZZ\mid\h^iX\ne0\}$.
We denote by $\dm(R)$ the derived category of complexes $X$ of $R$-modules such that $\h^iX=0$ for all $i\gg0$.

\begin{thm}\label{5}
Let $x\in\m$ be an $R$-regular element.
Then the following two statements hold.
\\
{\rm(1)} There is an inequality $\dx(R)\le2\dx(R/(x))+1$.\quad
{\rm(2)} One has that $\dx(R/(x))\le2\dx(R)+1$ if $x\notin\m^2$.
\end{thm}

\begin{proof}
Let $X\in\db(R)$.
Then $X\lten_RR/(x)$ is isomorphic to $X\otimes_R\k(x,R)$, which is a (homologically) bounded complex.
Hence $X\lten_RR/(x)\in\db(R/(x))$, and we get an exact functor $-\lten_RR/(x):\db(R)\to\db(R/(x))$.
This sends each object in $\thick_{\db(R)}R$ to an object in $\thick_{\db(R/(x))}R/(x)$.
By Lemma \ref{4}, we see that $-\lten_RR/(x)$ induces an exact functor $\Phi:\ds(R)\to\ds(R/(x))$.
On the other hand, as $\pd_RR/(x)=1<\infty$, all the objects in $\thick_{\db(R/(x))}R/(x)$ have finite projective dimension over $R$.
The exact functor $-\lten_{R/(x)}(R/(x))_R:\db(R/(x))\to\db(R)$ sends each object in $\thick_{\db(R/(x))}R/(x)$ to an object in $\thick_{\db(R)}R$.
We see from Lemma \ref{4} that $-\lten_{R/(x)}(R/(x))_R$ induces an exact functor $\Psi:\ds(R/(x))\to\ds(R)$.

(1) Pick any nonzero object $X$ of $\ds(R)$.
Then $X$ is an object of $\db(R)$ with $\pd_RX=\infty$.
We get
$$
\pd_{R/(x)}\Phi(X)=\pd_{R/(x)}(X\lten_RR/(x))=-\inf(X\lten_RR/(x)\lten_{R/(x)}k)=-\inf(X\lten_Rk)=\pd_RX=\infty
$$
by \cite[(A.5.7.2)]{C}.
Hence, $\Phi(X)$ is nonzero in $\ds(R/(x))$.
Now, let $n$ be an integer with $\dx(R/(x))\le n$.
Then $k$ is in $\langle\Phi(X)\rangle_{n+1}^{\ds(R/(x))}$.
The exact functor $\Psi$ provides $k=\Psi(k)\in\langle \Psi\Phi(X)\rangle_{n+1}^{\ds(R)}$.
Applying the exact functor $X\lten_R-$ to the exact triangle $R\xrightarrow{x}R\to R/(x)\rightsquigarrow$ in $\db(R)$, we get an exact triangle $X\xrightarrow{x}X\to\Psi\Phi(X)\rightsquigarrow$ in $\ds(R)$, which implies $\Psi\Phi(X)\in\langle X\rangle_2^{\ds(R)}$.
Thus $k$ belongs to $\langle X\rangle_{2(n+1)}^{\ds(R)}$, and it follows that $\dx(R)\le2n+1$.

(2) Pick any nonzero object $X$ of $\ds(R/(x))$.
There exist an $R/(x)$-module $M$ and an integer $m$ such that $X\cong M[m]$ in $\ds(R/(x))$ by \cite[Lemma 2.4(2a)]{sing}.
Then $M$ has infinite projective dimension over $R/(x)$.
As $x$ is $R$-regular and outside $\m^2$, we see from \cite[Theorem 2.2.3]{A} that $M$ has infinite projective dimension over $R$, and so does $N=\syz_RM$.
Hence $N$ is nonzero in $\ds(R)$.
Let $n$ be an integer with $\dx(R)\le n$.
Then $k$ belongs to $\langle N\rangle_{n+1}^{\ds(R)}$, and so does $\m=\syz_Rk$.
Since $x$ is $N$-regular, we have $\Phi(N)=N\lten_RR/(x)=N/xN$.
The exact functor $\Phi$ shows $\m/x\m=\Phi(\m)\in\langle\Phi(N)\rangle_{n+1}^{\ds(R/(x))}=\langle N/xN\rangle_{n+1}^{\ds(R/(x))}$.
As there is an $R/(x)$-isomorphism $\m/x\m\cong k\oplus\m/(x)$ (see \cite[Corollary 5.3]{syz2} for instance), $k$ belongs to $\langle N/xN\rangle_{n+1}^{\ds(R/(x))}$.
The chain map from the exact sequence $0\to N\to F\to M\to0$ with $F$ free to itself given by multiplication by $x$ induces an exact sequence $0\to M\to N/xN\to\syz_{R/(x)}M\to0$ of $R/(x)$-modules.
Hence $k\in\langle N/xN\rangle_{n+1}^{\ds(R/(x))}=\langle M\rangle_{2(n+1)}^{\ds(R/(x))}=\langle X\rangle_{2(n+1)}^{\ds(R/(x))}$, and we thus conclude that $\dx(R/(x))\le2n+1$.
\end{proof}

The above theorem gives the following, which includes a uniformly dominant version of \cite[Corollary 5.8]{dlr}.

\begin{cor}\label{11}
The local ring $R$ is uniformly dominant if and only if so is the formal power series ring $R[\![x]\!]$, if and only if so is the completion $\widehat R$.
More precisely, the following hold, where $m\ge0$ and $e=\edim R$.
\begin{enumerate}[\rm(1)]
\item
There are inequalities $\dx(R)\le2^m(\dx(R[\![x_1,\dots,x_m]\!])+1)-1\text{ and }\dx(R[\![x_1,\dots,x_m]\!])\le2^m(\dx(R)+1)-1$.
\item
There are inequalities $\dx(\widehat R)\le2^{2e}(\dx(R)+1)-1$ and $\dx(R)\le2^{2e}(\dx(\widehat R)+1)-1$.
\end{enumerate}
\end{cor}

\begin{proof}
(1) We may assume $m=1$.
The element $x_1$ is $R[\![x_1]\!]$-regular and does not belong to the square of the maximal ideal of $R[\![x_1]\!]$, and $R[\![x_1]\!]/(x_1)$ is isomorphic to $R$.
The assertion follows by Theorem \ref{5}.

(2) We can write $\m=(a_1,\dots,a_e)$.
The completion $\widehat R$ is isomorphic to $R[\![x_1,\dots,x_e]\!]/(x_1-a_1,\dots,x_e-a_e)$; see \cite[Theorem 8.12]{M}.
For every integer $1\le i\le e$, the element $x_i-a_i$ is regular on $R[\![x_1,\dots,x_e]\!]/(x_1-a_1,\dots,x_{i-1}-a_{i-1})$ and is outside of the square of its maximal ideal; see \cite[Lemma 5.7]{dlr}.
Therefore, 
$$
\dx(\widehat R)
=\dx(R[\![x_1,\dots,x_e]\!]/(x_1-a_1,\dots,x_e-a_e))
\le2^e(\dx(R[\![x_1,\dots,x_e]\!])+1)-1
\le2^{2e}(\dx(R)+1)-1,
$$
where the two inequalities follow from Theorem \ref{5}(2) and (1), respectively.
On the other hand, it holds that 
$$
\dx(R)
\le2^e(\dx(R[\![x_1,\dots,x_e]\!])+1)-1
\le2^{2e}(\dx(R[\![x_1,\dots,x_e]\!]/(x_1-a_1,\dots,x_e-a_e))+1)-1
=2^{2e}(\dx(\widehat R)+1)-1,
$$
where the two inequalities follow from (1) and Theorem \ref{5}(1), respectively.
\end{proof}

Next we state and prove a uniformly dominant version of \cite[Proposition 8.3]{dlr}.

\begin{prop}\label{8}
Let $R$ be Cohen--Macaulay with dimension one.
Let $I$ be an $\m$-primary ideal of $R$.
Let $a,b$ be parameters of $R$, i.e., $a,b$ are elements of $\m$ with $R/aR,R/bR$ artinian. 
Then $\dx(R/bI)\le4\dx(R/aI)+3$.
\end{prop}

\begin{proof}
Set $S=R[\![x]\!]/xIR[\![x]\!]$.
The proof of \cite[Proposition 8.3]{dlr} shows $x-a,x-b$ are $S$-regular elements, $x-b$ is not in the square of the maximal ideal of $S$, and $S/(x-c)\cong R/cI$ for any parameter $c$ of $R$.
Thus
$$
\dx(R/bI)
=\dx(S/(x-b))
\le2\dx(S)+1
\le4\dx(S/(x-a))+3
=4\dx(R/aI)+3,
$$
where the two inequalities are shown to hold by using (2) and (1) of Theorem \ref{5}, respectively.
\end{proof}

Now we obtain an example of a uniformly dominant local ring, which is a refinement of \cite[Corollary 8.4]{dlr}.

\begin{cor}\label{31}
Let $k$ be a field.
Let $a_1>a_2>\cdots>a_n=0=b_1<b_2<\cdots<b_n$ be integers with $n\ge3$.
Let $R=k[x,y]/(x^{a_1},x^{a_2}y^{b_2},\dots,x^{a_{n-1}}y^{b_{n-1}},y^{b_n})$.
Then $R$ is a uniformly dominant local ring with $\dx(R)\le15$.
\end{cor}

\begin{proof}
Let $S=k[\![x,y]\!]/(x^{a_1})$, and take its ideals $\n=(x,y)$ and $I=(x^{a_2},x^{a_3}y^{b_3-b_2},\dots,x^{a_{n-1}}y^{b_{n-1}-b_2},y^{b_n-b_2})$.
The proof of \cite[Corollary 8.4]{dlr} shows that $I$ is $\n$-primary, that $y$ is a parameter of $S$, that $S/yI$ is a singular artinian Burch ring (note that $a_1\ge2$ as $n\ge3$), and that $S/y^{b_2}I\cong R$.
We have $\dx(S/yI)\le1(2\cdot0+4)-1=3$ by virtue of Corollary \ref{40}(2).
Applying Proposition \ref{8}, we observe that $\dx(R)=\dx(S/y^{b_2}I)\le4\cdot3+3=15$.
\end{proof}

\begin{rem}
\begin{enumerate}[(1)]
\item
Recall from Corollary \ref{40} that every Burch ring and every local ring with quasi-decomposable maximal ideal are uniformly dominant local rings.
By combining Corollary \ref{31} with the last statement of \cite[Corollary 8.4]{dlr}, we see that, unless $a_r-a_{r+1}=1$ or $b_{r+1}-b_r=1$ for some $1\le r<n$, the ring $R$ in Corollary \ref{31} is a uniformly dominant local ring which is not Burch.
The artinian ring $R$ in Corollary \ref{31} is a uniformly dominant local ring whose maximal ideal is not (quasi-)decomposable, unless $n=3$ and $a_2=b_2=1$.
\item
By virtue of \cite[Theorem 1.3]{CDEKPU}, for an artinian non-Burch local ring $R$ of embedding dimension two, if $\m$ is indecomposable, then so is $\syz^2k$.
Combining this with Corollary \ref{31}, we obtain a lot of examples of a uniformly dominant local ring the second syzygy of whose residue field is indecomposable.
\end{enumerate}
\end{rem}

Next we establish a uniformly dominant version of \cite[Proposition 8.6]{dlr}.

\begin{prop}\label{12}
Let $A=k[\![x_1,\dots,x_h]\!]$ and $B=k[\![y_1,\dots,y_m]\!]$ with $k$ a field.
Let $f_1,\dots,f_l\in(x_1,\dots,x_h)^2$ and $g_1,\dots,g_h\in(y_1,\dots,y_m)$.
Put $R=A/(f_1,\dots,f_l)$ and $S=B/(\widetilde{f_1},\dots,\widetilde{f_l})$, where $\widetilde{f_i}=f_i(g_1,\dots,g_h)$ for each $i$.
Suppose that $x_1-g_1,\dots,x_h-g_h$ is a regular sequence on $R[\![y_1,\dots,y_m]\!]$ (this is satisfied if $R$ is Cohen--Macaulay and $\height(f_1,\dots,f_l)=\height(\widetilde{f_1},\dots,\widetilde{f_l})$).
Then $\dx(S)\le2^{m+h}(\dx(R)+1)-1$ and $\dx(R)\le2^{m+h}(\dx(S)+1)-1$.
\end{prop}

\begin{proof}
The proof of \cite[Proposition 8.6]{dlr} shows that $S\cong R[\![y_1,\dots,y_m]\!]/(x_1-g_1,\dots,x_h-g_h)$, that $x_i-g_i$ is not in the square of the maximal ideal of $R[\![y_1,\dots,y_m]\!]/(x_1-g_1,\dots,x_{i-1}-g_{i-1})$ for each $i$, and that if $R$ is Cohen--Macaulay and $\height(f_1,\dots,f_l)=\height(\widetilde{f_1},\dots,\widetilde{f_l})$, then $x_1-g_1,\dots,x_h-g_h$ is a regular sequence on $R[\![y_1,\dots,y_m]\!]$.
We have $\dx(S)\le2^h(\dx(R[\![y_1,\dots,y_m]\!])+1)-1\le2^{h+m}(\dx(R)+1)-1$ by Theorem \ref{5}(2) and Corollary \ref{11}(1).
Also, $\dx(R)\le2^m(\dx(R[\![y_1,\dots,y_m]\!])+1)-1\le2^{m+h}(\dx(S)+1)-1$ by Corollary \ref{11}(1) and Theorem \ref{5}(1).
\end{proof}

Now we get another example of a uniformly dominant local ring, which refines \cite[Corollary 8.7]{dlr}.
For a matrix $A$ over $R$ and an integer $r$, we denote by $\I_r(A)$ the ideal of $R$ generated by the $r\times r$ minors of $A$.

\begin{cor}\label{13}
Let $R=k[\![x_1,\dots,x_h]\!]/\I_2\!\big(\begin{smallmatrix}
f_{11}&\dots&f_{1u}\\
f_{21}&\dots&f_{2u}
\end{smallmatrix}\big)$, where $k$ is a field and each $f_{ij}$ is a nonunit of the local ring $k[\![x_1,\dots,x_h]\!]$.
If $\dim R=h-u+1$, then $R$ is uniformly dominant with $\dx(R)\le2^{h+8u+1}(u+3)-1$.
\end{cor}

\begin{proof}
Let $S=k\big[\begin{smallmatrix}
y_{11}&\dots&y_{1u}\\
y_{21}&\dots&y_{2u}
\end{smallmatrix}\big]/\I_2\!\big(\begin{smallmatrix}
y_{11}&\dots&y_{1u}\\
y_{21}&\dots&y_{2u}
\end{smallmatrix}\big)$ be a determinantal ring, and $\n$ its irrelevant maximal ideal.
The proof of \cite[Corollary 8.7]{dlr} shows $S_\n$ is a singular Cohen--Macaulay Burch ring of dimension $u+1$, and the defining ideal of $T:=\widehat{S_\n}$ has the same height as that of $R$.
It follows by Corollary \ref{40}(2) that $\dx(S_\n)\le a:=2^{2u}(2(u+1)+4)-1=2^{2u+1}(u+3)-1$.
Corollary \ref{11}(2) implies $\dx(T)\le b:=2^{4u}(a+1)-1=2^{6u+1}(u+3)-1$.
From the first inequality in Proposition \ref{12} we obtain $\dx(R)\le2^{h+2u}(b+1)-1=2^{h+8u+1}(u+3)-1$.
\end{proof}

Here we present a concrete example which can be obtained as an application of the above corollary.

\begin{ex}
Let $k$ be a field and $a,b,c>0$ integers.
Let $A=k[s^a,s^b,s^c]$ be the subring of the polynomial ring $k[s]$.
Let $R=k[\![s^a,s^b,s^c]\!]$ be the $(s^a,s^b,s^c)A$-adic completion of $A$.
Suppose that $R$ is not a complete intersection.
Then $R\cong k[\![x,y,z]\!]/\I_2\!\big(\begin{smallmatrix}
x^p&y^q&z^r\\
y^u&z^v&x^w
\end{smallmatrix}\big)$, where $p,q,r,u,v,w$ are positive integers; see \cite{H}. 
As $\dim R=1=3-3+1$, Corollary \ref{13} shows that $R$ is uniformly dominant with $\dx(R)\le2^{3+8\cdot3+1}(3+3)-1=3\cdot2^{29}-1$.
\end{ex}

Next we establish uniformly dominant versions of \cite[Propositions 7.4 and 8.8]{dlr}.

\begin{prop}\label{27}
Let $R$ be a regular local ring of dimension $d$.
\begin{enumerate}[\rm(1)]
\item
Let $S$ be a local ring which is faithfully flat over $R$.
Then one has $\dx(S)\le2^d(\dx(S/\m S)+1)-1$.
\item
Let $\xx=x_1,\dots,x_h$ and $\yy=y_1,\dots,y_m$ be sequences of elements in $\m$ such that $h>0$ and $m\ge0$.
Assume that the sequence $\xx,\yy=x_1,\dots,x_h,y_1,\dots,y_m$ is $R$-regular.
Then the following two statements hold true.
\begin{enumerate}[\rm(a)]
\item
If the integer $m$ is positive (i.e., nonzero), then there is an inequality $\dx(R/(\xx\yy))\le5\cdot2^{3d+2h+2m}-1$.
\item
If an element $z$ in $\m$ is $R/(\yy)$-regular, then one has the inequality $\dx(R/(\xx(\yy,z)))\le5\cdot2^{3d+2h+2m+2}-1$.
\end{enumerate}
\end{enumerate}
\end{prop}

\begin{proof}
(1) Let $\xx=x_1,\dots,x_d$ be a regular system of parameters of $R$.
Then $\xx$ generates the maximal ideal $\m$ of $R$ and is an $R$-regular sequence.
Since the homomorphism $R\to S$ is local and flat, we see that $\xx$ is an $S$-regular sequence.
Theorem \ref{5}(1) shows that $\dx(S)\le2^d(\dx(S/\xx S)+1)-1=2^d(\dx(S/\m S)+1)-1$.

(2b) Set $T=R/(\xx(\yy,z))$.
Let $A=\widehat{R}[\![\XX,\YY,Z]\!]/(\XX(\YY,Z))$ with indeterminates $\XX=X_1,\dots,X_h$, $\YY=Y_1,\dots,Y_m$, $Z$.
Then $A$ is faithfully flat over $R$, as $R\to A$ is local and $A$ is flat over the torsion-free (so, flat) $\ZZ$-algebra $\ZZ[\XX,\YY,Z]/(\XX(\YY,Z))$.
Let $B=A/\m A=k[\![\XX,\YY,Z]\!]/(\XX(\YY,Z))$.
The proof of \cite[Proposition 8.8]{dlr} shows that the sequence $\XX-\xx,\YY-\yy,Z-z$ is $A$-regular, that $A/(\XX-\xx,\YY-\yy,Z-z)$ is isomorphic to $\widehat T$, and that the maximal ideal of $B$ decomposes into $\XX B$ and $(\YY,Z)B$.
Lemma \ref{6} says $\depth B=\min\{\depth B/\XX B,\depth B/(\YY,Z)B,1\}=1$.
Corollary \ref{40}(1) implies $\dx(B)\le a:=2^{h+m+1}(2\cdot1+3)-1=5\cdot2^{h+m+1}-1$.
By (1) we get $\dx(A)\le b:=2^d(a+1)-1=5\cdot2^{d+h+m+1}-1$.
We see from Theorem \ref{5}(2) and \cite[Lemma 5.7]{dlr} that $\dx(\widehat T)\le c:=2^{h+m+1}(b+1)-1=5\cdot2^{d+2h+2m+2}-1$, and see from Corollary \ref{11}(2) that $\dx(T)\le2^{2d}(c+1)-1=5\cdot2^{3d+2h+2m+2}-1$.

(2a) Set $S=R/(\xx\yy)$.
The assertion is shown by ignoring everything on $Z$ or $z$ in the proof of (2b); we only give an outline.
The ring $A=\widehat{R}[\![\XX,\YY]\!]/(\XX\YY)$ is faithfully flat over $R$, the sequence $\XX-\xx,\YY-\yy$ is $A$-regular, $A/(\XX-\xx,\YY-\yy)$ is isomorphic to $\widehat S$, and the maximal ideal of the local ring $B=A/\m A=k[\![\XX,\YY]\!]/(\XX\YY)$ decomposes into $\XX B$ and $\YY B$.
The equality $\depth B=\min\{\depth B/\XX B,\depth B/\YY B,1\}=1$ holds since $m$ is assumed to be positive.
Therefore, we observe that $\dx(B)\le a:=2^{h+m}(2\cdot1+3)-1=5\cdot2^{h+m}-1$, that $\dx(\widehat S)\le2^{h+m}(\dx(A)+1)-1\le b:=2^{h+m}(2^d(a+1)-1+1)-1=5\cdot2^{d+2h+2m}-1$, and that $\dx(S)\le2^{2d}(b+1)-1=5\cdot2^{3d+2h+2m}-1$.
\end{proof}

We obtain another example of a uniformly dominant local ring, which improves \cite[Corollary 8.9]{dlr}.

\begin{cor}\label{32}
Let $R$ be a regular local ring of dimension $d$.
Let $I$ be an ideal of $R$ such that $I\subseteq\m^2$ and $\nu(I)\le2$.
Then the factor ring $R/I$ is either a local complete intersection with codimension two, or a uniformly dominant local ring such that the inequality $\dx(R/I)\le5\cdot2^{3d+6}-1$ holds.
\end{cor}

\begin{proof}
Put $n=5\cdot2^{3d+6}-1$.
If $\nu(I)=0$, then $I=0$, $R/I=R$ is regular, and therefore $R/I$ is uniformly dominant with $\dx(R/I)=-1\le n$.
Let $\nu(I)=1$.
Then $d>0$ and $R/I$ is a hypersurface.
Since $I$ is contained in $\m^2$, the local ring $R/I$ is singular.
Corollary \ref{40}(2) shows that $\dx(R/I)\le a$, where $a=3\le n$ if $d=1$, and $a=2^d(2(d-1)+4)-1=2^{d+1}(d+1)-1\le n$ if $d\ge2$.
Now let $\nu(I)=2$.
As is shown in the proof of \cite[Corollary 8.9]{dlr}, there exist a nonzero element $g$ of $R$ and an $R$-regular sequence $v,w$ such that $I=g(v,w)$.
If $g$ is not in $\m$, then $R/I$ is a complete intersection with codimension two.
If $g$ is in $\m$, then letting $\xx=v,w$ (hence $h=2$), $m=0$ and $z=g$ in Proposition \ref{27}(2b), we observe that $\dx(R/I)\le5\cdot2^{3d+2\cdot2+2\cdot0+2}-1=n$.
\end{proof}

We close the section by posing a rather natural question.

\begin{ques}
Does there exist a dominant local ring which is not uniformly dominant?
\end{ques}

\begin{ac}
The author thanks Souvik Dey and anonymous referees for giving him useful comments on earlier versions of this paper.
The author also thanks Naoyuki Matsuoka and Shunsuke Takagi for giving him helpful comments concerning \cite{H} and \cite{Ar,ST}, respectively.
\end{ac}


\end{document}